\documentclass[11pt,reqno]{amsart}

\usepackage[utf8]{inputenc} 

\usepackage[margin=1in]{geometry} 

\usepackage{graphicx} 
\usepackage{float} 

 \usepackage[parfill]{parskip} 
 
\usepackage{booktabs} 
\usepackage{array} 
\usepackage{paralist} 
\usepackage{verbatim} 
\usepackage{subfig} 
\usepackage{mathrsfs}
\usepackage{amssymb}
\usepackage{xcolor}
\usepackage{amsthm}
\usepackage{amsmath,amsfonts,amssymb,esint,hyperref}
\usepackage[noabbrev, capitalize]{cleveref}

\usepackage{autonum}

\usepackage{graphics,color}
\usepackage{enumerate, enumitem}
\usepackage{mathtools,centernot}
\usepackage{cases}
\usepackage{amsrefs}
\usepackage{bbm}
\usepackage{xfrac}

\usepackage{bookmark}

\newtheorem{theorem}{Theorem}[section]
\newtheorem{lemma}[theorem]{Lemma}

\newtheorem{corollary}[theorem]{Corollary}
\newtheorem{definition}[theorem]{Definition}
\newtheorem{proposition}[theorem]{Proposition}
\newtheorem{remark}[theorem]{Remark}

\numberwithin{equation}{section} 

\newcommand{\norm}[1]{\left\|#1\right\|}
\newcommand{\abs}[1]{\left|#1\right|}

\newcommand*{\R}{\ensuremath{\mathbb{R}}}
\renewcommand*{\H}{\ensuremath{\mathcal{H}}}

\newcommand*{\N}{\ensuremath{\mathbb{N}}}
\newcommand*{\Z}{\ensuremath{\mathbb{Z}}}

\newcommand{\eps}{\varepsilon}
\newcommand*{\tr}{\ensuremath{\mathrm{tr\,}}}

\newcommand{\e}{\varepsilon}

\newcommand{\quotes}[1]{``#1''}

\DeclareMathOperator{\Div}{Div}

\renewcommand{\MR}[1]{} 

\usepackage{color, graphicx}
\usepackage{mathrsfs, dsfont}

\usepackage[]{hyperref}
\hypersetup{
    colorlinks=true,       
    linkcolor=red,          
    citecolor=blue,        
    filecolor=red,      
    urlcolor=cyan           
}

\def\div{\mathop{\rm div}\nolimits}    
\def\disc{\mathop{\rm Disc}\nolimits} 

\def\spt{\mathop{\rm Spt}\nolimits} 
\def\tr{\mathop{\rm Tr}\nolimits} 
\def\Lip{\mathop{\rm Lip}\nolimits}  
 
\def\Div{\mathop{\rm Div}\nolimits}

\newcommand{\be}{\begin{equation}}
\newcommand{\ee}{\end{equation}}

\title{Normal traces and applications \\ to continuity equations on bounded domains}

\author[G. Crippa]{Gianluca Crippa}
\address[G. Crippa]{Departement Mathematik und Informatik, Universit\"at Basel, CH-4051 Basel, Switzerland}
\email{gianluca.crippa@unibas.ch}

\author[L. De Rosa]{Luigi De Rosa}
\address[L. De Rosa]{Departement Mathematik und Informatik, Universit\"at Basel, CH-4051 Basel, Switzerland}
\email{luigi.derosa@unibas.ch}

\author[M. Inversi]{Marco Inversi}
\address[M. Inversi]{Departement Mathematik und Informatik, Universit\"at Basel, CH-4051 Basel, Switzerland}
\email{marco.inversi@unibas.ch}

\author[M. Nesi]{Matteo Nesi}
\address[M. Nesi]{Departement Mathematik und Informatik, Universit\"at Basel, CH-4051 Basel, Switzerland}
\email{matteo.nesi@unibas.ch}

\date{\today}

\thanks{\textit{Acknowledgements}.
 The authors acknowledge the support of the SNF grant FLUTURA: Fluids, Turbulence, Advection No. 212573}

\begin{document}

\begin{abstract}
In this work, we study several properties of the normal Lebesgue trace of vector fields introduced  by the second and third author in \cite{DRINV23} in the context of the energy conservation for the Euler equations in Onsager-critical classes. Among other things, we prove that the normal Lebesgue trace satisfies the Gauss-Green identity and, by providing explicit counterexamples, that it is a notion sitting strictly between the distributional one for measure-divergence vector fields and the strong one for $BV$ functions. These results are then applied to the study of the uniqueness of weak solutions for continuity equations on bounded domains, allowing to remove the assumption in \cite{CDS14} of global $BV$ regularity up to the boundary, at least around the portion of the boundary where the characteristics exit the domain or are tangent. The proof relies on an explicit renormalization formula completely characterized by the boundary datum and the positive part of the normal Lebesgue trace. In the case when the characteristics enter the domain, a counterexample shows that achieving the normal trace in the Lebesgue sense is not enough to prevent non-uniqueness, and thus a $BV$ assumption seems to be necessary to get uniqueness.
\end{abstract}

\maketitle

\section{Introduction}

Throughout this note we will work in any spatial dimension $d\geq 2$. Before stating our main results, we start by recalling some definitions and explaining the main context. 

\begin{definition}[$L^p$ Measure-divergence vector fields]
    Let $1\leq p\leq \infty$ and let $\Omega \subset \R^d$ be an open set. Given a vector field $u:\Omega \rightarrow \R^d$ we say that\footnote{In the literature the notation $\mathcal{DM}^p(\Omega)$   is probably more common, i.e. \quotes{Divergence-measure} vector fields.} $u\in \mathcal{MD}^p(\Omega)$ if $u\in L^p(\Omega)$ and $\div u \in \mathcal M(\Omega)$.
\end{definition}
Here $\mathcal{M}(\Omega)$ denotes the space of finite measures over the open set $\Omega$. Building on an intuition by Anzellotti \cites{Anz83,Anz83bis}, a weak (distributional) notion of normal trace can be defined by imposing the validity of the Gauss-Green identity.

\begin{definition}[Distributional normal trace]\label{D:ditrib norm trace}
Let $\Omega \subset \R^d$ be a bounded open set with Lipschitz boundary. Given a vector field $u\in \mathcal{MD}^1(\Omega)$, and denoting $\lambda=\div u$, we define its outward distributional normal trace on $\partial \Omega$ by
    \begin{equation}\label{distr_norm_trace}
    \langle \tr_n (u;\partial\Omega), \varphi\rangle :=\int_\Omega\varphi \,d\lambda + \int_\Omega u\cdot \nabla \varphi\,dy\qquad \forall \varphi \in C^\infty_c(\R^d).
      \end{equation}
\end{definition}

Note that by a standard density argument, considering  $\varphi\in \Lip_c(\R^d)$ in \eqref{distr_norm_trace} yields to an equivalent definition. Clearly $\tr_n (u;\partial\Omega)$ is in general a distribution of order $1$. However, see for instance \cite{ACM05}*{Proposition 3.2}, in the case $u\in \mathcal{MD}^\infty(\Omega)$ the distributional trace is in fact induced by a measurable function and $\tr_n (u;\partial\Omega)\in L^\infty (\partial \Omega;\mathcal H^{d-1})$. Note also that we are adopting the convention that $n:\partial \Omega\rightarrow \mathbb{S}^{d-1}$ is the \emph{outward} unit normal vector, which will be kept through the whole manuscript.

Measure-divergence vector fields, with particular emphasis on the case $p=\infty$, have received  great attention in recent years. They happen to be very useful in various contexts such as establishing fine properties of vector fields with bounded deformation \cite{ACM05}, existence and uniqueness for continuity-type equations with a physical boundary \cites{CDS14,CDS14bis},  conservation laws \cites{CTZ09,CT11,CTZ07,CT05,CF99}, dissipative anomalies and intermittency in turbulent flows \cite{DDI23} and several others.  For a detailed analysis of the  theoretical properties and refinements of $\mathcal{MD}^p$ we refer the interested reader to \cites{PT08,CCT19,Sil23,Sil05} and references therein.

\subsection{The normal Lebesgue trace}
The downside of the generality of \cref{D:ditrib norm trace} is in that it does not prevent bad behaviours of the vector field $u$ in the proximity of $\partial \Omega$. For instance in \cite{CDS14} it has been shown that having a distributional trace of the vector field is not enough to guarantee uniqueness for transport and/or continuity equations on a domain $\Omega$ with boundary, even when a boundary datum is properly assigned. On the positive side, in the same paper \cite{CDS14} it has been also shown that a $BV$ assumption on $u$ up to the boundary does imply uniqueness, the main reason being the fact that $BV$ functions achieve their traces in a sufficiently strong sense (see \cref{S:tools} for the definition and main properties of $BV$ functions). More recently, a new  notion of \emph{normal Lebesgue boundary trace} has been introduced in \cite{DRINV23} in the context of the energy conservation for the Euler equations in Onsager-critical classes. We recall it here.

\begin{definition}[Normal Lebesgue boundary trace]\label{D:Leb normal trace}
Let $\Omega\subset \R^d$ be a bounded open set with Lipschitz boundary and let $u\in L^1(\Omega)$ be a vector field. We say that $u$ admits an inward Lebesgue normal trace on $\partial \Omega$ if there exists a function  $f\in L^1(\partial \Omega;\mathcal H^{d-1})$ such that, for every sequence $r_k\rightarrow  0$, it holds
$$ \lim_{k \to \infty}\frac{1}{r_k^d} \int_{B_{r_k} (x)\cap\Omega} \left|(u\cdot \nabla d_{\partial \Omega})(y)- f(x)\right|\,dy=0 \qquad \text{for } \mathcal{H}^{d-1}\text{-a.e. } x\in \partial \Omega. $$
Whenever such a function exists, we will denote it by $f=: u_{-n}^{\partial \Omega}$. Consequently, the outward Lebesgue normal trace will be $u_{n}^{\partial \Omega}:= - u_{-n}^{\partial \Omega}$. 
\end{definition} 

It is an easy check that, whenever it exists, the normal Lebesgue trace is unique (see \cite{DRINV23}*{Section 5.1}). Note that here, to keep consistency with $\tr_n(u;\partial\Omega)$ being the outward normal distributional trace, we are adopting the opposite convention with respect to \cite{DRINV23}*{Definition 5.2} by switching the sign. The above definition has been used in the context of weak solutions to the incompressible Euler equations to prevent the energy dissipation from happening at the boundary \cite{DRINV23}*{Theorem 1.3}. It has also been proved (see the proof of \cite{DRINV23}*{Proposition 5.5}) that for $u\in BV$ the normal Lebesgue boundary trace exists, with explicit representation with respect to the \emph{full} trace of the vector field. Clearly, the definition of the normal Lebesgue trace can be restricted to any measurable subset $\Sigma\subset \partial \Omega$. This will be done in \cref{S:NLT local and signed}, together with the study of further properties which will be important for the application to the continuity equations on bounded domains. Due to the very weak regularity of the objects involved, our analysis requires several technical tools from geometric measure theory and in particular establishes properties of sets with Lipschitz boundary that might be interesting in themselves.

From now on we restrict ourselves to bounded vector fields, since otherwise the distributional normal trace might fail to be induced by a function. Our first main result shows that, for $u\in \mathcal{MD}^\infty$, if the normal Lebesgue trace exists it must coincide with the distributional one. In particular, it satisfies the Gauss-Green identity. We emphasize that, in general, the theorem below fails if $u$ is not bounded (see \cref{R:adolfo}).

\begin{theorem}\label{T:strong trace equals weak}
Let $\Omega \subset \R^d$ be a bounded open set with Lipschitz boundary and let $u\in \mathcal{MD}^\infty(\Omega)$. Assume that $u$ has a normal Lebesgue trace on $\partial \Omega$ in the sense of \cref{D:Leb normal trace}. Then, it holds 
$u_n^{\partial\Omega}\equiv \tr_n (u;\partial\Omega)$ as elements of $L^\infty(\partial \Omega;\mathcal H^{d-1})$.
\end{theorem}
In particular, for bounded measure-divergence vector fields, either the normal Lebesgue trace does not exist or it exists and it coincides with the distributional one.
This observation will be used in \cref{l: example} to construct $u\in \mathcal{MD}^\infty$ which does not admit a normal Lebesgue trace. Moreover it is rather easy to find vector fields which do admit a normal Lebesgue trace but fail to be of bounded variation (see \cref{R: BV vs normal trace}). It follows that \cref{D:Leb normal trace} is a notion  lying strictly between the distributional one of \cref{D:ditrib norm trace} and the strong one for $BV$ vector fields (see \cref{T:trace_in_BV}). In \cref{S:NLT properties} we also prove some additional properties of the normal Lebesgue trace which might be of independent interest. Let us point out that here all the results assume the vector field to be bounded.

\begin{remark}\label{R:adolfo}
In the first version of this manuscript, we raised the question whether the existence of the normal Lebesgue trace implies $u_n^{\partial\Omega}= \tr_n (u;\partial\Omega)$   in the more general case $u \in \mathcal{MD}^1(\Omega)$.  Soon after posting the first version, A. Arroyo-Rabasa \cite{AAR_comm} provided us the following example that shows that the answer to the question is negative in general. In the (open) upper half two-dimensional ball $\Omega:=B_1(0)\cap \R^2_+$ consider $u(x)=x \abs{x}^{-2}$. Then $\div u = 0$ in $\Omega$, $u\in \mathcal{MD}^1(\Omega)$, $u_n^{\partial\Omega}\big|_{x_2=0}\equiv 0$ but $\tr_n (u;\partial\Omega)$ has a Dirac delta in the point $(0,0)\in \partial \Omega$. In particular,  $u^{\partial\Omega}_n$ does not satisfy the Gauss-Green identity. This example appeared already in \cite{W57}*{Section III.14} and has been further studied in \cites{CCT19,CF03,CIT24}.
\end{remark}

\subsection{Applications to the continuity equation} The groundbreaking theory of \emph{renormalized solutions} by DiPerna--Lions \cite{DipLi89} and Ambrosio \cite{Ambr04} establishes the well-posedness for weak solutions of the continuity equation with rough vector fields on the whole space $\R^d$, more precisely, in the case of Sobolev and $BV$ vector fields respectively. See for instance  \cites{CripDel08,Cr09} for a review.

Let us now focus on the bounded domain setting. Since we will be working with merely bounded solutions, the rigorous definitions become quite delicate. For this reason we postpone to \cref{S:Cont Eq} the main technical formulations which guarantee that all the objects involved are well defined.
We consider a solution $\rho:\Omega\times (0,T)\rightarrow \R$ to
\begin{equation}\label{CE}
    \left\{
    \begin{array}{rcll}
        \partial_t \rho + \div (\rho u  ) &=& c \rho +f & \text{ in }\Omega\times (0,T) \\
        \rho &=&  g & \text{ on }  \Gamma^- \\
        \rho(\cdot,0) &=&  \rho_0 & \text{ in }  \Omega,
    \end{array}
    \right.
\end{equation}
where $\Omega\subset \R^d$ is an open bounded set with Lipschitz boundary, $u:\Omega\times (0,T)\rightarrow \R^d$ is a given vector field, $\Gamma^- \subset \partial \Omega \times (0,T)$ is the (possibly time-dependent) portion of the boundary in which the characteristics are entering,  while $\rho_0:\Omega\rightarrow \R$, $g:\Gamma^-\rightarrow \R$ and $c,f: \Omega\times (0,T)\rightarrow \R$ are given data. Note that if $\rho$ and $u$ are  not sufficiently regular their value on negligible sets is not well defined. However, as noted in \cite{CDS14bis}, a distributional formulation of the problem \eqref{CE} can still be given by relying on the theory of measure-divergence vector fields described above, see \cref{D:Cont Eq weak sol}.  The existence of such weak solutions has been proved in \cite{CDS14bis} by parabolic regularization under quite general assumptions. A much more delicate issue is the uniqueness of weak solutions, which has been established in \cite{CDS14} under the assumption that $u\in L^1_{\rm loc}([0,T);BV(\Omega))$, that is when the vector field enjoys $BV$ regularity up to the boundary. The uniqueness result heavily relies on a suitable chain-rule formula for the normal trace of $\rho^2 u$ at the boundary, previously established in \cite{ACM05}*{Theorem 4.2}, which holds when $u\in BV(\Omega)$. Our main goal is to show that no $BV$ assumption on $u$ is necessary around the portion of the boundary where the characteristics exit, as soon as a suitable behaviour in terms of the normal Lebesgue trace is assumed. 

In the next theorem, for a set $A\subset \partial \Omega$, we will denote its $r$-tubular neighbourhood \quotes{interior to $\Omega$} by $(A)_r^{\rm in}:=(A)_r\cap \Omega$, where $(A)_r$ is the standard tubular neighbourhood (in $\R^d$) of width $r>0$. We refer to \cref{S:Notation} for a more detailed guideline on the notation used in this whole note.

\begin{theorem}
    \label{T:CE general}
    Let $\Omega\subset \R^d$ be a bounded open set with Lipschitz boundary and $u \in L^\infty(\Omega\times (0,T))\cap L^1_{\rm loc}([0,T);BV_{\rm loc}(\Omega))$ be a vector field such that $\div u \in L^1((0,T);L^\infty(\Omega))$. Let 
 $\Gamma^-,\Gamma^+\subset \partial\Omega\times (0,T)$ be as in \cref{D:Gamma plus and minus} and assume that 
 \begin{itemize}
     \item[(i)] there exists an open set $O \subset \R^{d}\times (0,T)$ such that $\Gamma^- \subset O$, $u_t \in BV_{\rm loc}(O_t)$ for a.e. $t\in (0,T)$ and $\nabla u_t \otimes dt \in \mathcal{M}_{\rm loc} (O)$,
     \item[(ii)] denoting by $\Gamma^+_t\subset \partial \Omega$ the $t$-time slice of the space-time set $\Gamma^+$, for a.e. $t\in (0,T)$ it holds  
     \begin{equation}
         \label{hp: gamma out exit}
         \lim_{r\rightarrow 0} \frac{1}{r} \int_{(\Gamma^+_t)_{r}^{\rm in}} (u_t\cdot \nabla d_{\partial \Omega})_+ \,dx=0.
     \end{equation}
 \end{itemize} 
Moreover, let $f\in L^1(\Omega\times (0,T))$, $c\in L^1((0,T);L^\infty(\Omega))$, $\rho_0 \in L^\infty (\Omega)$ and $g \in L^\infty(\Gamma^-)$ be given. Then, in the class $\rho \in L^\infty(\Omega\times (0,T))$, the problem \eqref{CE} admits at most one distributional solution in the sense of \cref{D:Cont Eq weak sol}.
\end{theorem}

\begin{remark}[Existence]
\cref{T:CE general} is concerned only with uniqueness. The existence part,  in the case $\div u, f, c\in L^\infty(\Omega\times (0,T))$ is more classical and can be found in \cite{CDS14bis}. Notice that the generalization to the case $f\in L^1(\Omega\times (0,T))$, $c\in L^1((0,T);L^\infty(\Omega))$ directly follows by a standard truncation argument together with the a priori bound on the solution.
\end{remark}

A more general version of the above theorem will be given in \cref{S:Cont Eq} (see \cref{T:CE particular} and \cref{C: TE tangent}) where we prove that an explicit weak formulation for $\beta(\rho)$ holds up to the boundary, for any $\beta\in C^1(\R)$, that is the vector field $u$ satisfies a renormalization property on $\Omega\times [0,T)$. The renormalization property can be seen in a certain sense as an analogue in the linear case of the conservation of the energy studied in \cite{DRINV23} in the context of the Euler equations. In particular, it is natural to investigate the role of the normal Lebesgue trace for the renormalization property.  The assumption \eqref{hp: gamma out exit} can be thought of as a way to force characteristics to be \quotes{uniformly} exiting, which we will show to be enough to prevent non-uniqueness phenomena. Equivalently, \eqref{hp: gamma out exit} dampens any \quotes{recoil} of the vector field which could cause mass to enter the domain $\Omega$ around the portion of $\partial \Omega$ where on average (i.e. in the weak sense) it points outward, a phenomenon which relates to ill-posedness. More effective conditions in terms of the normal Lebesgue trace from \cref{D:Leb normal trace} which imply \eqref{hp: gamma out exit} will be given in \cref{S:NLT local and signed} (see for instance \cref{C:NLT pos and neg}). As a consequence of our general results on the relation between the normal Lebesgue trace and the distributional one, in \cref{P: BV and exit} we will show that \eqref{hp: gamma out exit} holds true as soon as $u$ is $BV$ up to the boundary, while in general it is a strictly weaker assumption. In some sense, \cref{T:CE general} shows that the subset of the boundary in which characteristics are entering, i.e. $\Gamma^-$, is more problematic than $\Gamma^+$ since it requires the vector field to be $BV$ in its neighbourhood. Indeed, we notice that in the counter-example  built in \cite{CDS14} the vector field achieves the normal boundary trace in the strong Lebesgue sense.

\begin{proposition}\label{P:counter-ex}
 Let $\Omega:=   \R^2\times (0, +\infty) $. There exists an autonomous vector field $u:\Omega\rightarrow \R^3$ such that $\div u=0$, $u\in L^\infty(\Omega)\cap BV_{\rm loc} (\Omega)$, $\tr_n (u;\partial\Omega)\equiv u^{\partial\Omega}_n\equiv -1$ and the initial-boundary value problem 
 \begin{equation}\label{CE counter ex}
    \left\{
    \begin{array}{rcll}
        \partial_t \rho + u\cdot \nabla \rho &=& 0 & \text{ in }\Omega\times (0,1) \\
        \rho &=&  0 & \text{ on }  \partial \Omega\times (0,1) \\
        \rho(\cdot,0) &=&  0 & \text{ in }  \Omega
    \end{array}
    \right.
\end{equation}
admits infinitely many weak solutions in the sense of \cref{D:Cont Eq weak sol}.
\end{proposition}
The reader may notice that the domain $\Omega$ in the above proposition is unbounded. This choice has been made for convenience in order to directly consider the Depauw-type construction from \cite{CDS14}*{Proposition 1.2}, so that the vector field $u$ can enter on the full boundary $\partial \Omega =\R^2\times \{0\}$ while still being incompressible. This is clearly enough to show that a non-trivial $\Gamma^-$ makes both notions of traces from \cref{D:ditrib norm trace} and \cref{D:Leb normal trace} not sufficient to obtain well-posedness. Furthermore, let us mention that  also \eqref{hp: gamma out exit} cannot be avoided. Indeed in \cite{CDS14}*{Theorem 1.3} the authors construct an autonomous vector field with $\tr_n(u;\partial\Omega)\equiv 1$ which fails to guarantee uniqueness. As discussed in \cite{CDS14}, such construction can be also modified to have $\tr_n( u;\partial \Omega)\equiv 0$. Thus, in the context considered here, the assumptions made in \cref{T:CE general} are essentially optimal and they single out the behaviour of rough vector fields which is truly relevant.

\subsection{Plan of the paper} \cref{S:tools} contains all the technical tools: we start by introducing the main notation, then we recall some basic facts about weak convergence of measures and $BV$ functions and we conclude by proving some technical results about the convergence of  Minkowski-type contents.  In \cref{S:NLT properties} we focus on various properties of the normal Lebesgue trace: we prove the Gauss-Green identity in \cref{T:strong trace equals weak}, the convergence of the positive and negative parts of the Lebesgue trace, the connection with $BV$ vector fields and conclude with \cref{l: example} by constructing  a vector field which admits a distributional normal trace but fails to have the Lebesgue one. The last \cref{S:Cont Eq} contains all the applications to the continuity equation: after recalling the main setting from \cite{CDS14}, which allows to define weak solutions, in \cref{T:CE particular} we prove the main well-posedness result, which is a more general version of \cref{T:CE general}, and then conclude with the proof of \cref{P:counter-ex}.

\section{Technical tools}\label{S:tools}
In this section we collect some, mostly measure theoretic, tools which will be needed. We start by introducing some notation.

\subsection{Notation}\label{S:Notation}

\begin{itemize}
    \item we set $\mathbb{G}_m^d = \left\{ \text{linear $m$-dimensional subspaces in } \R^d \right\}$ and we denote by $\Pi\in \mathbb{G}_m^d$ its elements;
    \item $\omega_m$ is the $m$-volume of the $m$-dimensional unit ball;
    \item given an open set $\Omega \subset \R^d$ and $T>0$, we set $\Lambda:=\partial \Omega\times (0,T)$ and  $\mathcal{L}_{\partial \Omega}^T := \left(\mathcal{H}^{d-1} \otimes dt \right)\llcorner \Lambda$; 
    \item given a bounded vector field $u:\Omega\times (0,T)\rightarrow \R^d$, we denote by $\Gamma^+,\Gamma^-\subset \Lambda$ the parts of the boundary in which characteristics are exiting and entering respectively (see \cref{D:Gamma plus and minus}); 
    \item for any space-time measurable set $A\subset \R^d\times (0,T)$ we denote by $A_t:= \{ x\in \R^d\,:\,(x,t)\in A\}$ its slice at a given time $t$;
    \item given $f:\Omega\times (0,T)\rightarrow \R$ we denote by $f_t:\Omega\rightarrow \R$ the map $f_t(\cdot):=f(\cdot ,t)$;
    \item given a set $M \subset \R^d$, for any point $x\in \R^d$, we define $d_M(x) = \inf \{ \abs{x-y}\, \colon \,y \in M \}$;
    \item given a closed set $M\subset \R^d$, we denote by $\pi_M:\R^d\rightarrow M$ the projection map onto $M$;
    \item given $r>0$ and a set $M \subset \R^d$, we denote by $(M)_r := \{ x \in \R^d \,\colon\, d_M(x) < r\}$  the open tubular neighbourhood of radius $r$; 
    \item given an open set $\Omega \subset \R^d$ and $r>0$, we denote by $(\partial \Omega)_r^{\rm in} := (\partial \Omega)_r \cap \Omega$ and $(\partial \Omega)_r^{\rm out} := (\partial \Omega)_r \cap (\R^d \setminus \overline{\Omega})$ the \emph{interior} and \emph{exterior} tubular neighbourhoods of $\partial \Omega$ respectively;
    \item slightly abusing notation, when $A\subset \partial\Omega$ we still denote by $(A)_r^{\rm in }:= (A)_r\cap \Omega$ and $(A)_r^{\rm out }:= (A)_r\cap (\R^d\setminus \overline\Omega)$ the parts of the $r$-tubular neighbourhoods of $A$ which belong to $\Omega$ and $\Omega^c$ respectively;
    \item given an open set $\Omega$ we denote by $\mathcal M(\Omega)$ the space of finite signed measures on $\Omega$, while $\mathcal M_{\rm loc} (\Omega)$ denotes the space of Radon measures on $\Omega$;
    \item given $\mu \in \mathcal{M}(\Omega)$, we denote by $\spt \mu$ the support of the measure $\mu$, that is the smallest closed set where $\mu$ is concentrated;
    \item $\tr_n(u;\partial \Omega)$ is the distributional normal trace from \cref{D:ditrib norm trace};
    \item  $u^{\partial\Omega}_n$ is the normal Lebesgue trace from \cref{D:Leb normal trace};
    \item  $u^{\Sigma}_n$ is the normal Lebesgue trace from \cref{D:NLT local} on a subset $\Sigma\subset\partial \Omega$;
    \item $u^\Omega$ is the full trace  of $u$ on $\partial \Omega$, in the $BV$ sense, as defined in \cref{T:trace_in_BV};
    \item whenever we consider the space-time set $\Omega\times (0,T)$, we denote by $n_\Omega$ the outer normal to $\partial\Omega$;
    \item for any function $f$ we denote by $f_+$ and $f_-$ its positive and negative part respectively, i.e. $f=f_+-f_-$;
    \item we denote by $\div$ the divergence with respect to the spatial variable;
    \item we denote by $\Div$ the space-time divergence, that is $\Div (u,f)=\div u+\partial_t f$, where $u:\Omega\times (0,T)\rightarrow\R^d$ and $f:\Omega\times (0,T)\rightarrow\R$. 
\end{itemize}

\subsection{Weak convergence of measures} Here we recall some basic facts on weak convergence of measures.

\begin{definition} \label{d: weak convergence of measures}
Let $\mu_k, \mu \in \mathcal M_{\rm loc}(\R^d)$. We say that $\{\mu_k\}_k$ converges weakly to $\mu$, denoted by $\mu_k \rightharpoonup \mu$, if for any test function $\varphi \in C_c(\R^d)$ it holds  
\begin{equation}
    \lim_{k \to \infty} \int_{\R^d} \varphi \, d \mu_k = \int_{\R^d} \varphi \, d \mu. \label{eq: weak convergence test}
\end{equation}
\end{definition}

Weak convergence of measures can be characterized as follows. 

\begin{proposition} \label{p:char of weak convergence}
Let $\mu_k, \mu\in \mathcal M(\R^d)$ be such that $ \mu_k(\R^d) \rightarrow \mu(\R^d)$.
The following facts are equivalent: 
\begin{itemize}
    \item $\mu_k \rightharpoonup \mu$ according to \cref{d: weak convergence of measures}; 
    \item for any open set $U \subset \R^d$ it holds $\mu(U) \leq \liminf_{k \to \infty} \mu_k(U)$.
\end{itemize}
\end{proposition}

\begin{proof}
It is immediate to see that the convergence of the total mass together with the lower semicontinuity on open sets imply  $\mu(C) \geq \limsup_{k \to \infty} \mu_k(C)$ for all $C \subset \R^d$  closed. It is well known that having both lower semicontinuity on open sets and upper semicontinuity on compact sets is equivalent to weak convergence, see for instance \cite{EG15}*{Theorem 1.40}.
\end{proof}

The following proposition is part of the so-called \quotes{Portmanteau theorem} (see for instance \cite{K08}*{Theorem 13.16} for a proof).
\begin{proposition} \label{p: almost continuous test function}
Let $\mu_k,\mu\in \mathcal M(\R^d)$ be such that 
\begin{equation}\label{tight_convergence}
   \lim_{k\rightarrow \infty} \int_{\R^d} \varphi \,d\mu_k=\int_{\R^d} \varphi \,d\mu \qquad \forall \varphi \in C^0_b(\R^d).
\end{equation}
Let $f: \R^d \to \R$ be a bounded  Borel function and denote by $\disc (f)$ the set of its discontinuity points.  If $\mu(\disc (f))=0$ then
\begin{equation}
    \lim_{k \to \infty} \int_{\R^d} f \, d \mu_k = \int_{\R^d} f \, d \mu. \label{eq: testing with a.e. continuous}
\end{equation}
\end{proposition}

\subsection{Functions of bounded variation} Let $\Omega \subset \R^d$ be an open set. We say that $f \in L^1(\Omega)$ is a function of \emph{bounded variation} if its distributional gradient  is represented by a finite  measure on $\Omega$, i.e. we set $BV(\Omega) = \{ f \in L^1(\Omega) \,\colon\, \nabla f \in \mathcal{M}(\Omega) \}$.
An $m$-dimensional vector field $f:\Omega\rightarrow \R^m$ is said to be of bounded variation if all its components are $BV$ functions.  The space of vector fields with bounded variation will be denoted by $BV(\Omega;\R^m)$, or, slightly abusing the notation, simply by $BV(\Omega)$ when no confusion can occur. We refer to the monograph \cite{AFP00} for an extensive discussion of the theory of $BV$ functions. Here we only 
recall from \cite{AFP00}*{Theorem 3.87} that a $BV$ vector field on a Lipschitz domain admits a notion of trace on the boundary. 

\begin{theorem}[Boundary trace]\label{T:trace_in_BV}
Let $\Omega\subset \R^d$ be a bounded open set with  Lipschitz boundary and $f\in BV(\Omega;\R^m)$. There exists $f^\Omega\in L^1(\partial \Omega;\mathcal H^{d-1})$ such that 
$$ \lim_{r\rightarrow  0}\frac{1}{r^d} \int_{B_r (x)\cap\Omega} \left|f(y)- f^\Omega (x)\right|\,dy=0\qquad \text{for } \mathcal{H}^{d-1}\text{-a.e. } x\in \partial \Omega. $$
Moreover, the extension $\tilde f$ of $f$ to zero outside $\Omega$ belongs to $BV(\R^d;\R^m)$ and
$$
\nabla \tilde f=(\nabla f)\llcorner\Omega - (f^\Omega\otimes n) \mathcal{H}^{d-1} \llcorner\partial \Omega,
$$
being $n:\partial \Omega\rightarrow \mathbb{S}^{d-1}$ the outward unit normal.
\end{theorem}
 
\begin{remark}[$BV$ vs. normal trace]\label{R: BV vs normal trace}
In \cite{DRINV23}*{$(ii)$-Proposition 5.5} it has been proved that if $u\in BV(\Omega)\cap L^\infty(\Omega)$ then the outward normal Lebesgue trace $u^{\partial\Omega}_n$ exists and it is given by $u^{\partial\Omega}_n=u^\Omega \cdot n$, where $u^\Omega$ is the full $BV$ trace on $\partial \Omega$ of the vector field $u$ and $n:\partial \Omega\rightarrow \R^d$ is the outward unit normal to $\partial \Omega$. Note that in \cite{DRINV23}*{$(ii)$-Proposition 5.5} the assumption $u\in L^\infty$ is not necessary and $u\in BV(\Omega)$ is enough. On the other hand, it is easy to find  $u\not\in BV(\Omega)$ which admits a normal Lebesgue trace. Indeed, on $\Omega=\R^2_+$, the vector field $u(x,y)=(g(y),0)$ always satisfies $u^{\partial\Omega}_n\equiv 0$ (and moreover $\div u=0$) but it is not $BV_{\rm loc}(\R^2_+)$ as soon as $g\not \in BV_{\rm loc}(\R_+)$.
\end{remark}

We also recall the standard DiPerna--Lions \cite{DipLi89} and Ambrosio \cite{Ambr04} commutator estimate. For any function $f$ we denote by $f_\e$ its mollification. 

\begin{lemma} \label{L:commutator}
Let $O \subset \R^d$ be an open set, $u \in BV_{\rm loc}(O)$ be a vector field and $\rho \in L^\infty(O)$. Then, for every compact set $K \subset \joinrel \subset O$ it holds 
    \[
        \limsup_{\eps\to 0 } \norm{u \cdot \nabla \rho_\eps-\div(\rho  u)_\eps}_{L^1(K)}
                \leq C \norm{\rho}_{L^\infty(O)} \abs{\nabla u}(K ).
    \]
\end{lemma}

\subsection{Slicing and traces}

We recall the following property from the slicing theory of Sobolev functions. The trace of a Sobolev function on a bounded open set with Lipschitz boundary is defined according to \cref{T:trace_in_BV}. Since we were not able to find a reference, we also give the (simple) proof.  

\begin{proposition}
\label{p: slicing and traces}
Let $\Omega \subset \R^d$ be a bounded open set with Lipschitz boundary and let $ u \in W^{1,1}(\Omega \times (0,T))$. Then, for a.e. $t \in (0,T)$ it holds that $u(\cdot, t) \in W^{1,1}(\Omega)$ and $u(\cdot, t)^{\Omega} = u^{\Omega \times (0,T)}(\cdot, t)$ as functions in $L^1(\partial \Omega; \mathcal{H}^{d-1})$. 
\end{proposition} 

\begin{proof}
Since $u \in W^{1,1}(\Omega \times (0,T))$, by Fubini theorem we can assume that $u(\cdot, t), \nabla u(\cdot, t) \in L^1(\Omega)$ and $u^{\Omega \times (0,T)}(\cdot, t) \in L^1(\partial \Omega)$ for almost every $t \in (0,T)$. Then, for any $\alpha \in C^\infty_c((0,T)), \varphi \in C^1_c(\R^d)$, we have  
\begin{equation}
    \int_0^T \alpha(t) \left(  \int_{\Omega} u(x,t) \partial_i \varphi (x) \, dx + \int_{\Omega} \partial_i u(x,t) \varphi(x) \, dx - \int_{\partial \Omega} \varphi(x) u^{\Omega\times (0,T)}(x,t) n_i (x) \, d\mathcal{H}^{d-1} \right) \, dt = 0. 
\end{equation}
Therefore, we find a negligible set of times $\mathcal{N}_{\varphi} \subset (0,T)$ such that for any $t \in (0,T) \setminus \mathcal{N}_\varphi $ it holds 
\begin{equation}
    \int_{\Omega} u(x,t) \partial_i \varphi (x) \, dx = - \int_{\Omega} \partial_i u(x,t) \varphi(x) \, dx + \int_{\partial \Omega} \varphi(x) u^{\Omega\times (0,T)}(x,t) n_i(x) \, d\mathcal{H}^{d-1}. \label{eq: slicing 1}
\end{equation}
Letting $\mathcal{D} \subset C^{1}_c(\R^d)$ countable and dense, we find a negligible set of times $\mathcal{N}\subset (0,T)$ such that \eqref{eq: slicing 1} holds for any $t \in (0,T) \setminus \mathcal{N}$ and for any $\varphi \in \mathcal{D}$. Thus, by a standard approximation argument, \eqref{eq: slicing 1} is valid for any $t \in (0,T) \setminus \mathcal{N}$ and for any $\varphi \in C^1_c(\R^d)$. Hence, given $t \in (0,T) \setminus \mathcal{N}$, we have that $u(\cdot, t) \in W^{1,1}(\Omega)$ and $u(\cdot, t)^{\Omega} = u^{\Omega \times (0,T)} (\cdot, t)$ as functions in $L^1(\partial \Omega)$. 
\end{proof}

\subsection{Measure-divergence vector fields}

Here we provide some basic facts on gluing and multiplications of measure-divergence vector fields. More general statements can be found in \cite{CCT19}.

\begin{lemma} \label{l: gluing lemma}
Let $\Omega \subset \R^d$ be a bounded open set with Lipschitz boundary. Let $u$ be a vector field in $L^\infty(\R^d)$. Assume that $u \in \mathcal{MD}^\infty(\Omega) \cup \mathcal{MD}^\infty(\R^d \setminus \overline{\Omega})$.  Then, $u \in \mathcal{MD}^\infty(\R^d)$ and it holds
\begin{equation}
    \div u = (\div u)\llcorner \Omega + (\div u) \llcorner (\R^d \setminus \overline{\Omega}) - \left(\tr_n(u; \partial \Omega) + \tr_n(u; \partial (\R^d \setminus \overline{\Omega})) \right) \mathcal{H}^{d-1} \llcorner \partial \Omega. \label{eq: formula gluing}
\end{equation} 
\end{lemma}

\begin{proof}
Denote by $(\div u) \llcorner \Omega = \mu_1 \in \mathcal{M}(\Omega)$ and $( \div u) \llcorner (\R^d \setminus \overline{\Omega}) = \mu_2 \in \mathcal{M}(\R^d \setminus \overline{\Omega})$. Given a test function $\varphi \in C^\infty_c(\R^d)$, by \eqref{distr_norm_trace} we get 
\begin{align}
    \langle \div u, \varphi \rangle & = - \int_{\Omega} u\cdot \nabla \varphi \, dx - \int_{\R^d \setminus \overline{\Omega}} u \cdot \nabla \varphi \, dx 
    \\ & = \int_{\Omega} \varphi \, d \mu_1 + \int_{\R^d \setminus \overline{ \Omega}} \varphi \, d\mu_2 - \left \langle \tr_n(u; \partial \Omega) + \tr_n(u; \partial (\R^d \setminus \overline \Omega)) , \varphi \right\rangle, 
\end{align}
thus proving \eqref{eq: formula gluing}, since $\tr_n(u; \partial \Omega), \tr_n(u; \partial (\R^d \setminus \overline \Omega)) \in L^\infty( \partial \Omega)$.
\end{proof}

\begin{lemma} \label{l: product rule}
Let $\Omega \subset \R^d$ be a bounded open set with Lipschitz boundary. Let $h \in W^{1,1}(\Omega) \cap L^\infty(\Omega)$ be a scalar function and $u \in L^{\infty}(\Omega)$ be a vector field such that $\div u \in L^1(\Omega)$. Then, $h u \in \mathcal{MD}^{\infty}(\Omega)$ and 
\begin{equation} \label{eq: product rule divergence}
    \div(hu) = h \div u + u \cdot \nabla h, 
\end{equation}
\begin{equation} \label{eq: product rule trace}
    \tr_n(h u; \partial \Omega) = h^{ \Omega} \tr_n(u; \partial \Omega), \
\end{equation}
where $h^{\Omega}$ is the trace of $h$ on $\partial \Omega$ in the sense of Sobolev functions (see \cref{T:trace_in_BV}).  
\end{lemma}

\begin{proof}
Both \eqref{eq: product rule divergence} and \eqref{eq: product rule trace} are trivial if $h\in C^\infty(\R^d)$. Then, for $h \in W^{1,1}(\Omega) \cap L^\infty(\Omega)$, recalling that $\Omega$ is bounded with Lipschitz boundary, we find a sequence $\{h_\e\}_\e \subset C^{\infty}_c(\R^d)$ such that $h_\e \to h$ strongly in $W^{1,1}(\Omega)$, $\norm{h_\e}_{L^\infty(\Omega)} \leq \norm{h}_{L^\infty(\Omega)}$ and $h_\e(x) \to h(x)$ for a.e. $x \in \Omega$. Moreover, since the trace operator is continuous from $W^{1,1}(\Omega)$ to $L^1(\partial \Omega)$, we infer that $h_\e^{ \Omega} \to h^{ \Omega}$ strongly in $L^1( \Omega)$. Then, writing \eqref{eq: product rule divergence} and \eqref{eq: product rule trace} for $h_\e$, it is straightforward to pass to the limit as $\e \to 0$. 
\end{proof}

\subsection{Distance function and projection onto a closed set} Let $M \subset \R^d$ be a closed set. We recall that the distance function $d_M$ is Lipschitz continuous and thus almost everywhere differentiable on $\R^d$. 

\begin{lemma} \label{l:projection onto boundary}
Let $M$ be a closed set and define 
$$D= \left\{ x \in \R^d \,\colon \exists !\, y\in M \text{ s.\,t. } d_M(x)=|x-y|\right\}. $$
Then, $\mathcal H^d(D^c)=0$ and $M \subset D$. Let $\pi_M : D \to M$ be the projection onto $M$ that associates to $x$ the unique point of minimal distance. Then $\pi_{M}(x) = x$ for any $x \in M$. Moreover, if $\{x_j\}_j\subset D$ is any sequence  converging to $x \in M$,  it holds $\lim_{j \to \infty} \pi_{M} (x_j) = x$. 
In particular, if $f: M \to \R$ is a continuous function, then $f \circ \pi_M: D \to \R$ is continuous at any point in $M$. 
\end{lemma}

\begin{proof}
Following \cite{DRINV23}*{Lemma 2.3}, we have that $d_M$ is differentiable a.e. in $\R^d$ and for any point $x$ of differentiability it holds 
\begin{equation}
    \nabla d_M(x) = \frac{x-y}{\abs{x-y}}, \label{eq: gradient distance}  
\end{equation}
where $y \in M$ is any point of minimal distance between $x$ and $M$. We claim that $y$ is uniquely determined. Indeed, if there were $y_1, y_2 \in M$ minimizing the distance between $x$ and $M$, since $\abs{x-y_1}=\abs{x-y_2}$, it is clear that $y_1=y_2$ by \eqref{eq: gradient distance}. Thus $D$ is of full measure and the projection operator $\pi_M : D \to M$ is well defined at every point in $D$. Moreover, it is clear that $M \subset D$ and $\pi_M(x) = x$ for any $x \in M$. Lastly, letting $\{x_j\}_j$ be a sequence in $D$ converging to $x \in M$, by the minimality property of $\pi_M$ we get
\begin{equation}
    \lim_{j \to \infty}\abs{ \pi_M(x_j) - x } \leq \lim_{j \to \infty}\abs{\pi_M (x_j) - x_j} + \abs{x_j - x} \leq 2 \lim_{j \to \infty}\abs{x_j-x} = 0. 
\end{equation}
\end{proof}

\subsection{Rectifiable sets, Minkowski content, and Hausdorff measure} We define rectifiable sets according to \cite{Fed}*{Definition 3.2.14}. 

\begin{definition} \label{d: rectifiable sets}
We say that $M \subset \R^d$ is countably $m$-rectifiable if $M = \bigcup_{i \in \N} M_i$, where $M_i = f_i(E_i)$, $E_i \subset \R^m$ is a bounded Borel set and $f_i: \R^m \to \R^d$ is a Lipschitz map. 
\end{definition}

We recall the following property of the Minkowski content of a closed countably rectifiable set. Here $(M)_r$ denotes the open tubular neighborhood of radius $r>0$ (see \cref{S:Notation}).

\begin{proposition} [\cite{Fed}*{Theorem 3.2.39}] \label{p: minkowski federer}
Let $M \subset \R^d$ be a compact countably $m$-rectifiable set according to \cref{d: rectifiable sets}. Then 
\begin{equation}
    \lim_{r \to 0} \frac{\mathcal{H}^d((M)_r)}{\omega_{d-m} r^{d-m} } = \mathcal{H}^{m}(M). \label{eq: mink vs haus}
\end{equation}   
\end{proposition}

The left hand side of \eqref{eq: mink vs haus} is usually referred to as the \emph{Minkowski content} of $M$. We will mostly focus on the case $m=d-1$ and notice that $\omega_1=2$. We also recall the following characterization of the Hausdorff measure in terms of projections onto linear subspaces. 

\begin{proposition} [\cite{AFP00}*{Proposition 2.66}] \label{p: char of hausdorff}
Let $M \subset \R^d$ be a countably $m$-rectifiable set according to \cref{d: rectifiable sets}. Then 
$$\mathcal{H}^m(M) = \sup \left\{ \sum_{i=1}^n \mathcal{H}^m( \pi_{\Pi_i}(M_i) ) \,\colon\, \{\Pi_i\}_i \subset \mathbb{G}_m^d, \ M_i\subset M \text{ pairwise disjoint compact sets}\right\}. $$
\end{proposition}

Building on \cref{p: minkowski federer} and \cref{p: char of hausdorff}, we study the blow up of the Lebesgue measure around a closed $m$-rectifiable set. The proof of the following result is inspired by that of \cite{AFP00}*{Proposition 2.101}. 

\begin{proposition} \label{p: weak convergence to hausdorff}
Let $M$ be a compact countably $m$-rectifiable set in $\R^d$ according to \cref{d: rectifiable sets}. Assume that $\mathcal{H}^{m}(M)< +\infty$. It holds that $\displaystyle \frac{\mathcal{H}^d \llcorner (M)_r}{\omega_{d-m} r^{d-m} }  \rightharpoonup \mathcal{H}^{m} \llcorner M $ according to \cref{d: weak convergence of measures}.
\end{proposition}

\begin{proof} 
By \cref{p: minkowski federer} and \cref{p:char of weak convergence}, it is enough to prove that for any open set $O \subset \R^d$ it holds 
\begin{equation}
    \mathcal{H}^{m} (O \cap M) \leq \liminf_{r \to 0} \frac{\mathcal{H}^d(O \cap (M)_r)}{\omega_{d-m} r^{d-m}} =:\mathscr{M}^m_*(M; O).  
\end{equation}
 Moreover, using \cref{p: char of hausdorff}, it is enough to check that for any finite collection of pairwise disjoint compact sets $M_1, \dots, M_n\subset M\cap O$ and for any $\Pi_1, \dots, \Pi_n \in \mathbb{G}_m^d$ it holds that 
\begin{equation}
    \mathscr{M}^m_*(M; O) \geq \sum_{i=1}^n \mathcal{H}^{m}(\pi_i(M_i)), \qquad \text{with } \,\pi_i := \pi_{\Pi_i}.\label{eq: char of lower minkowski} 
\end{equation}
Since $M_1, \dots, M_n$ are compact and disjoint, it follows $\mathscr{M}^m_*(M; O) \geq\sum_{i=1}^n \mathscr{M}^m_*(M_i; O)$. Then, to prove \eqref{eq: char of lower minkowski} it suffices  to check that 
\begin{equation}
    \mathscr{M}^m_*(M_i; O) \geq \mathcal{H}^{m}(\pi_i(M_i))\qquad \forall i= 1,\dots, n. \label{eq: projection vs hausdorff}
\end{equation}
Given any $i \in \{1, \dots, n\}$, by Fubini's theorem and Fatou's lemma we compute 
\begin{align}
    \mathscr{M}^m_*(M_i; O) & = \liminf_{r \to 0} \frac{\mathcal{H}^d ((M_i)_r \cap O) }{\omega_{d-m} r^{d-m}}  
    \\ & = \liminf_{r \to 0} \int_{\Pi_i} \frac{\mathcal{H}^{d-m} \left( (M_i)_r \cap O \cap \pi_i^{-1}(x) \right)}{\omega_{d-m} r^{d-m} } \, d\mathcal{H}^{m}(x) 
    \\ & \geq \liminf_{r \to 0} \int_{\pi_i(M_i)}  \frac{\mathcal{H}^{d-m} \left( (M_i)_r \cap O \cap \pi_i^{-1}(x) \right)}{\omega_{d-m} r^{d-m}} \, d\mathcal{H}^{m}(x)
    \\ & \geq \int_{\pi_i (M_i) } \liminf_{r \to 0} \frac{\mathcal{H}^{d-m} \left( (M_i)_r \cap O \cap \pi_i^{-1}(x) \right)}{\omega_{d-m} r^{d-m}} \, d\mathcal{H}^{m}(x). 
\end{align}
Moreover, for any $x \in \pi_i(M_i)$ there exists $p_x \in M_i$ such that $\pi_i(p_x)=x$, and so $B_r(p_x) \subset (M_i)_r $. Since $M_i \subset O$ and $O$ is an open set, if $r$ is small enough (possibly depending on $x$), we get $B_r(p_x) \subset (M_i)_r \cap O$. Thus, for any $x \in \pi_i(M_i)$ it holds 
$$\pi_i^{-1}(x) \cap B_r(p_x) \subset \pi_i^{-1}(x) \cap (M_i)_r \cap O, $$
i.e. $\pi_i^{-1}(x) \cap (M_i)_r \cap O$ contains a $(d-m)$-dimensional ball of radius $r$, provided that $r$ is small enough (see \cref{fig:proj}). Hence, we conclude
$$\liminf_{r \to 0} \frac{\mathcal{H}^{d-m} \left( (M_i)_r \cap O \cap \pi_i^{-1}(x) \right)}{\omega_{d-m} r^{d-m}} \geq  1 \qquad \forall x \in \pi_i(M_i), $$
thus proving \eqref{eq: projection vs hausdorff}.
\begin{figure}
\includegraphics[width=0.6\textwidth]{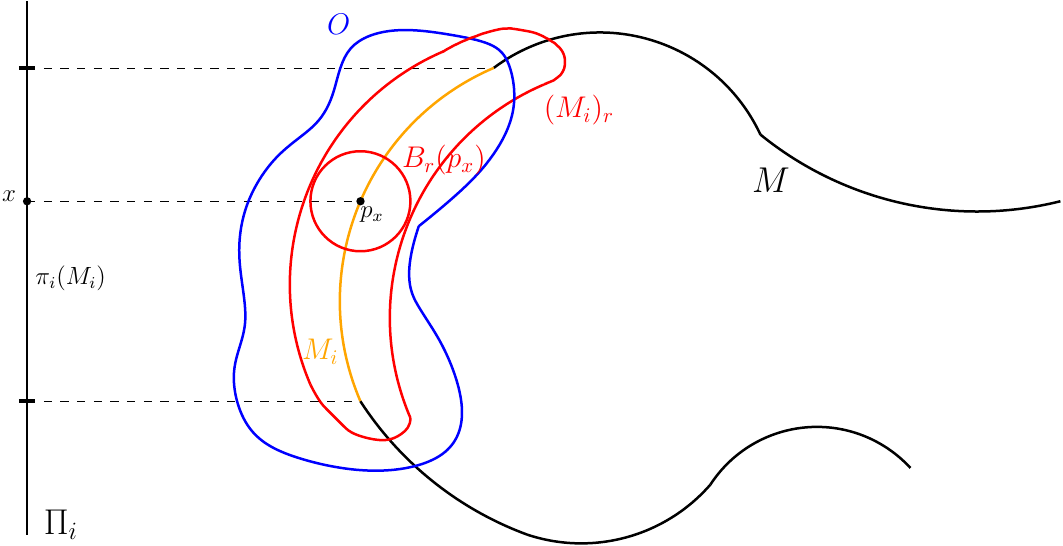}
\caption{Proof of \cref{p: weak convergence to hausdorff} in the case $d=2$ and $m=1$.}
\label{fig:proj}
\end{figure}
\end{proof}

\subsection{Lipschitz sets and one-sided Minkowski contents} We start by recalling a basic property of Lipschitz sets.  

\begin{lemma} \label{l: lipschitz set vs boundary cone property}
Let $\Omega \subset \R^d$ be a bounded open set with Lipschitz boundary. Then, for $\mathcal{H}^{d-1}$ almost every $x \in \partial \Omega$ there exists a unit vector $n_x$ such that for any $\alpha \in (0, \sfrac{\pi}{2})$ there exists $r_\alpha >0$ such that for any $r \in (0, r_\alpha)$ it holds that
\begin{equation}
    \{ y \in B_r(x)\, \colon \,\langle y-x, -n_x \rangle >\cos(\alpha) \abs{y-x} \} \subset \Omega, \label{eq: inside cone}
\end{equation}
\begin{equation}
    \{ y \in B_r(x)\, \colon \,\langle y-x, n_x \rangle > \cos(\alpha) \abs{y-x} \} \subset \R^d \setminus \overline{\Omega}. \label{eq: outside cone} 
\end{equation}
\end{lemma}

\begin{remark}
Given an open set $\Omega$ with Lipschitz boundary, \cref{l: lipschitz set vs boundary cone property} establishes the existence, $\mathcal{H}^{d-1}\llcorner\partial\Omega$ almost everywhere, of a unit vector $n_x$ such that for any $\alpha \in (0, \sfrac{\pi}{2})$ the cones of angle $\alpha$ around $-n_x$ and $n_x$ are contained in $\Omega$ and $\R^d \setminus \overline{\Omega}$ for small radii, respectively. Moreover, the vector $n_x$ is unique at any point at which it is defined and it plays the role of an outer unit normal vector.       
\end{remark}

Following the lines of the proof of \cref{p: weak convergence to hausdorff}, we establish the following result. 

\begin{proposition} \label{p: weak convergence to hausdorff one side}
Let $\Omega \subset \R^d$ be a bounded  open  set with Lipschitz boundary and $\Sigma \subset \partial \Omega$ Borel. Then, for any $O\subset \R^d$ open it holds
    \begin{equation}
        \label{int mink content lower semicont}
        \liminf_{r\rightarrow 0 } \frac{\mathcal{H}^d\left( (\Sigma)_r^{\rm in} \cap O\right)}{r}\geq \mathcal{H}^{d-1}(\Sigma\cap O)
    \end{equation}
    and
    \begin{equation}
        \label{ext mink content lower semicont}
        \liminf_{r\rightarrow 0 } \frac{\mathcal{H}^d \left((\Sigma)_r^{\rm out} \cap O\right)}{r}\geq \mathcal{H}^{d-1}(\Sigma\cap O).
    \end{equation}
\end{proposition}

\begin{proof}
Fix an open set $O \subset \R^d$.  We check the validity of \eqref{int mink content lower semicont} by following the proof of \cref{p: weak convergence to hausdorff}. The proof of \eqref{ext mink content lower semicont} is analogous and thus left to the reader. By \cref{p: char of hausdorff}, it is enough to check that for any finite collection of pairwise disjoint compact sets $M_1, \dots, M_n\subset \Sigma \cap O$ and for any $\Pi_1, \dots, \Pi_n \in \mathbb{G}_{d-1}^d$ it holds that 
\begin{equation}
    \liminf_{r\to 0} \frac{\mathcal{H}^{d} \left( (\Sigma)_r^{\rm in} \cap O \right) }{r} \geq \sum_{i=1}^n \mathcal{H}^{d-1}(\pi_i(M_i)), 
\end{equation}
where we set $\pi_i = \pi_{\Pi_i}$. Since $M_1, \dots, M_n$ are compact and disjoint, it is easy to check that 
\begin{equation}
    \liminf_{r \to 0} \frac{\mathcal{H}^{d} \left( (\Sigma)_r^{\rm in} \cap O \right) }{r} \geq \sum_{i=1}^n \liminf_{r \to 0} \frac{\mathcal{H}^{d} \left( (M_i)_r^{\rm in}  \cap O \right) }{r}. 
\end{equation} 
Then, to prove \eqref{int mink content lower semicont} it suffices to check that 
\begin{equation}
    \liminf_{r \to 0} \frac{\mathcal{H}^{d} \left( (M_i)_r^{\rm in}  \cap O \right) }{r} \geq \mathcal{H}^{d-1}(\pi_i(M_i))\qquad \forall i= 1,\dots, n. 
\end{equation}
Given any $i \in \{1, \dots, n\}$, by Fubini's theorem and Fatou's lemma we get
\begin{align}
    \liminf_{r \to 0} \frac{\mathcal{H}^{d} \left( (M_i)_r^{\rm in}  \cap O \right) }{r} & = \liminf_{r \to 0} \int_{\Pi_i} \frac{\mathcal{H}^{1} \left( (M_i)_r^{\rm in} \cap O \cap \pi_i^{-1}(x) \right)}{ r } \, d\mathcal{H}^{d-1} (x) 
    \\ & \geq \liminf_{r \to 0} \int_{\pi_i(M_i)}  \frac{\mathcal{H}^{1} \left( (M_i)_r^{\rm in} \cap O \cap \pi_i^{-1}(x) \right)}{r} \, d\mathcal{H}^{d-1}(x)
    \\ & \geq \int_{\pi_i(M_i) } \liminf_{r \to 0} \frac{\mathcal{H}^{1} \left( (M_i)_r^{\rm in} \cap O \cap \pi_i^{-1}(x) \right)}{r} \, d\mathcal{H}^{d-1}(x).
\end{align}
To conclude, it is enough to prove that for $\mathcal{H}^{d-1}$ almost every $x \in \pi_i(M_i)$ it holds that 
\begin{equation}
    \liminf_{r \to 0} \frac{\mathcal{H}^{d-1} \left( (M_i)_r^{\rm in} \cap O \cap \pi_i^{-1}(x) \right)}{r} \geq  1. \label{eq: half fiber}
\end{equation}
With the notation of \cref{l: lipschitz set vs boundary cone property}, we set
\begin{equation}
    S_i := \left\{ x \in \pi_i(M_i) \,\colon\, \exists p_x\in M_i \text{ s.t. $n_{p_x}$ is well defined and } n_{p_x} \notin \Pi_i \right\}. \label{eq: S_i}  
\end{equation}
We claim that $\mathcal{H}^{d-1}\left( \pi_i(M_i) \cap S_i^c\right) = 0$ and that \eqref{eq: half fiber} is satisfied for any $x \in S_i$.
To begin, we notice that  $\pi_i(M_i) \cap S_i^c \subset \pi_i( A_i) \cup \pi_i( B_i)$, with
$$
A_i = \{ y \in M_i \,\colon\, n_y \text{ is not defined} \} \qquad \text{and} \qquad B_i = \{ y \in M_i\, \colon\, n_y \text{ is defined and } n_y \in \Pi_i \}. 
$$
Since $\pi_i$ is $1$-Lipschitz and $n_y$ is defined for $\mathcal{H}^{d-1}$-a.e. $y \in \partial \Omega$, we infer that $\mathcal{H}^{d-1}(\pi_i(A_i)) \leq \mathcal{H}^{d-1}(A_i)=0$.
Next, we prove that $\mathcal{H}^{d-1}(\pi_i (B_i)) = 0$. Recalling that $B_i$ is $(d-1)$-rectifiable, by the area formula with the tangential differential \cite{AFP00}*{Theorem 2.91}, we have  
\begin{equation}
    \int_{\Pi_i} \mathcal{H}^0 \left(B_i \cap \pi_i^{-1}(y)\right) \, d\mathcal{H}^{d-1}(y) = \int_{B_i} J_{d-1}^{B_i} \pi_i(y) \, d \mathcal{H}^{d-1}(y).  
\end{equation}
Here $J_{d-1}^{B_i} \pi_i(y)$ is the determinant of the differential of the restriction of $\pi_i$ to $y + \text{Tan}(y; B_i)$, computed at $y$. We notice that 
$$\mathcal{H}^0 \left(B_i \cap \pi_i^{-1}(y)\right) \geq \mathds{1}_{\pi_i(B_i)}(y) \qquad \forall y \in \Pi_i, $$
thus proving 
\begin{equation}
    \mathcal{H}^{d-1}(\pi_i(B_i)) \leq \int_{B_i} J_{d-1}^{B_i} \pi_i(y) \, d \mathcal{H}^{d-1}(y). 
\end{equation}
Moreover, for any $y \in B_i$, $\pi_i$ is constant along any line contained in the tangent space to $\partial \Omega$ at $y$. Thus, the determinant of the tangential Jacobian at $y$ vanishes. Therefore, we deduce
$$\int_{B_i} J_{d-1}^{B_i} \pi_i(y) \, d \mathcal{H}^{d-1}(y) = 0,$$
yielding $\mathcal{H}^{d-1}(\pi_i(B_i)) = 0$. 
To conclude, pick any $x \in S_i$. We check that \eqref{eq: half fiber} is satisfied at $x$. Let $v_i$ be a unit vector such that $\Pi_i^{\perp} = \text{Span}(v_i)$. Since $x\in S_i$ we can find $p\in M_i$ such that $\pi_i(p)=x$ and $n_{p} \notin \Pi_i$. Without loss of generality we can assume that $\langle n_{p}, v_i \rangle >0$. Then, we can find an angle $\alpha_{p} \in (0, \sfrac{\pi}{2})$ such that $0 < \cos(\alpha_{p}) < \langle n_{p}, v_i \rangle$.
Letting $r_{\alpha_{p}}$ as in \cref{l: lipschitz set vs boundary cone property}, it is clear that for any $r < r_{\alpha_p}$ the segment between $p$ and $p+ r v_i$ is contained in $ (M_i)^{\rm in}_r \cap \pi_i^{-1}(x)$. Since $O$ is an open set, the segment is also contained in $O$, possibly choosing a smaller  $r_{\alpha_p}$. This proves \eqref{eq: half fiber} at $x$. 
\end{proof}

Now, let us restrict to $\Sigma\subset \partial\Omega$ closed. We remark that by \cite{ACA08}*{Corollary 1} (see also the more general statement \cite{ACA08}*{Theorem 5}) we also have the convergence of the total masses of the two sequences of measures defined as
$$
\frac{\mathcal{H}^d\left((\Sigma)_r^{\rm in}\cap A\right)}{r} \qquad \text{and} \qquad \frac{\mathcal{H}^d\left((\Sigma)_r^{\rm out}\cap A\right)}{r} \qquad \qquad \forall A\subset \R^d \text{ Borel}.
$$
Thus, with \cref{p: weak convergence to hausdorff one side} in hand, by \cref{p:char of weak convergence} we could directly conclude the weak convergence of the one-sided Minkowski contents as measures concentrated on $\Sigma$.  However, in order to keep this note self-contained, the next corollary gives an independent and elementary proof of this fact.

\begin{corollary}\label{C:conv one side mink}
Let $\Omega \subset \R^d$ be a bounded open set with Lipschitz boundary and $\Sigma\subset \partial \Omega$ closed. Then, as $r\rightarrow 0$, it holds that 
$$
\frac{\mathcal{H}^d\llcorner(\Sigma)_r^{\rm in}}{r}\rightharpoonup \mathcal{H}^{d-1}\llcorner \Sigma \qquad \text{and} \qquad \frac{\mathcal{H}^d\llcorner(\Sigma)_r^{\rm out}}{r}\rightharpoonup \mathcal{H}^{d-1}\llcorner \Sigma.
$$
\end{corollary}
\begin{proof}
We want to apply \cref{p:char of weak convergence}. Thanks to \cref{p: weak convergence to hausdorff one side} we already have that both sequences of measures are lower semicontinuous on open sets. Thus, it suffices to check that their masses converge to $\mathcal{H}^{d-1}(\Sigma)$. 
We split 
$$
\frac{\mathcal H^d \left((\Sigma)_{r}\right)}{2{r}}= \frac{1}{2} \left(\frac{\mathcal H^d \left((\Sigma)_{r}^{\text{in}}\right)}{{r}} + \frac{\mathcal H^d \left((\Sigma)_{r}^{\text{out}}\right)}{{r}}\right).
$$
By \cref{p: minkowski federer} and by applying \eqref{int mink content lower semicont} and \eqref{ext mink content lower semicont} with $O=\R^d$ we deduce
$$
\mathcal{H}^{d-1}(\Sigma)\leq \frac{1}{2}\left( \liminf_{r\rightarrow 0} \frac{\mathcal H^d \left((\Sigma)_{r}^{\text{in}}\right)}{{r}} +\liminf_{r\rightarrow 0}  \frac{\mathcal H^d \left((\Sigma)_{r}^{\text{out}}\right)}{{r}}\right)\leq \mathcal{H}^{d-1}(\Sigma).
$$
In particular
$$
\liminf_{r\rightarrow 0} \frac{\mathcal H^d \left((\Sigma)_{r}^{\text{in}}\right)}{{r}} +\liminf_{r\rightarrow 0}  \frac{\mathcal H^d \left((\Sigma)_{r}^{\text{out}}\right)}{{r}}=2\mathcal{H}^{d-1}(\Sigma),
$$
which, by using again  \eqref{int mink content lower semicont} and \eqref{ext mink content lower semicont}, necessarily implies
$$
\liminf_{r\rightarrow 0}  \frac{\mathcal H^d \left((\Sigma)_{r}^{\text{in}}\right)}{{r}}=\mathcal{H}^{d-1}(\Sigma) \qquad \text{and}\qquad \liminf_{r\rightarrow 0}  \frac{\mathcal H^d \left((\Sigma)_{r}^{\text{out}}\right)}{{r}}=\mathcal{H}^{d-1}(\Sigma).
$$
Since the above inferior limits are uniquely defined and do not depend on the choice of the sequence $r\rightarrow 0$, we conclude
$$
\lim_{r\rightarrow 0}  \frac{\mathcal H^d \left((\Sigma)_{r}^{\text{in}}\right)}{{r}}=\mathcal{H}^{d-1}(\Sigma) \qquad \text{and}\qquad \lim_{r\rightarrow 0}  \frac{\mathcal H^d \left((\Sigma)_{r}^{\text{out}}\right)}{{r}}=\mathcal{H}^{d-1}(\Sigma).
$$
\end{proof}

\section{Normal Lebesgue trace: Gauss-Green  and further properties}\label{S:NLT properties}
In this section we prove several properties of the normal Lebesgue trace, the most important being the Gauss-Green identity. In addition to their possible independent interest, such properties will be used in the proof of \cref{T:CE general} and for a comparison with the previous results obtained in \cite{CDS14}.

\subsection{Gauss-Green identity}

Here we prove \cref{T:strong trace equals weak},  together with several others properties relating integrals on tubular neighbourhoods to boundary integrals of traces, when the latter are suitably defined. Everything will follow from the next general proposition.
\begin{proposition}\label{P:traces and tub neigh general}
Let $f:\Omega\rightarrow \R$, $f\in L^\infty(\Omega)$, $\Omega\subset \R^d$ a bounded open set with Lipschitz boundary and $\Sigma\subset \partial \Omega$ closed. Assume that there exists $f^\Sigma:\Sigma\rightarrow \R$ such that 
\begin{equation}\label{eq:trace general}
 \lim_{r\rightarrow  0}\frac{1}{r^d} \int_{B_r (x)\cap\Omega} \left|f(y)- f^\Sigma (x)\right|\,dy=0\qquad \text{for } \mathcal{H}^{d-1}\text{-a.e. } x\in \Sigma.
 \end{equation}
 Then, for any $\varphi\in C^\infty_c(\R^d)$  it holds 
 \begin{equation}
     \lim_{r\rightarrow 0}\frac{1}{r}\int_{(\Sigma)_r^{\rm in}}f\varphi \,dy=\int_\Sigma f^\Sigma \varphi \,d\mathcal H^{d-1}.
 \end{equation}
\end{proposition}
\begin{remark}\label{R:trace is bounded}
If $f\in L^\infty(\Omega)$, the sequence of functions 
$$
\Sigma\ni x\mapsto \frac{1}{r^d}\int_{B_r(x)\cap\Omega} f(y)\,dy
$$
is bounded in $L^\infty(\Sigma;\mathcal H^{d-1})$. Thus, $f^\Sigma \in L^\infty(\Sigma;\mathcal H^{d-1})$.
\end{remark}
A direct corollary of \cref{P:traces and tub neigh general} is the following.
\begin{corollary}\label{P:strong trace convergence}
Let $u:\Omega\rightarrow \R^d$, $u\in L^\infty(\Omega)$, $\Omega\subset \R^d$ a bounded open set with Lipschitz boundary and set
  \begin{equation} \label{eq:cut_off}
\chi_r (y):= \begin{cases}
    1 & y \in \Omega \setminus (\partial \Omega)_r^{\rm in}
    \\\frac{d_{\partial \Omega}(y)}{r} & y \in (\partial \Omega)_r^{\rm in}. 
\end{cases}
\end{equation}
Assume that $u$ has an outward normal Lebesgue trace $u^{\partial \Omega}_n$ on $\partial \Omega$ according to \cref{D:Leb normal trace}.
Then, for any $\varphi\in C^\infty_c(\R^d)$, it holds
\begin{equation}\label{converg_to_strong_trace}
\lim_{r\rightarrow 0} \int_\Omega \varphi u\cdot \nabla \chi_r  \,dy=-\int_{\partial \Omega} 
 \varphi u^{\partial \Omega}_n \,d\mathcal H^{d-1}.
 \end{equation}
\end{corollary}
\begin{proof}
    The left-hand side in \eqref{converg_to_strong_trace} can be written as 
    $$
    \lim_{r\rightarrow 0} \frac{1}{r}\int_{(\partial \Omega)_r^{\rm in}} \varphi u\cdot \nabla d_{\partial\Omega}  \,dy.
    $$
    Thus, by applying \cref{P:traces and tub neigh general} with $f=u\cdot \nabla d_{\partial \Omega}$, $f^\Sigma=u^{\partial \Omega}_{-n}$ and $\Sigma =\partial \Omega$, we obtain
$$
\lim_{r\rightarrow 0} \int_\Omega \varphi u\cdot \nabla \chi_r  \,dy=\int_{\partial \Omega} 
 \varphi u^{\partial \Omega}_{-n} \,d\mathcal H^{d-1}.
$$
The proof is concluded since $u^{\partial \Omega}_{-n}=-u^{\partial \Omega}_{n}$.
\end{proof}
Then, \cref{T:strong trace equals weak} directly follows.
\begin{proof}[Proof of \cref{T:strong trace equals weak}]
Denote $\lambda :=\div u$. Since $\varphi \chi_r \in \Lip_c(\R^d)$ and $\varphi \chi_r \big|_{\partial \Omega}\equiv 0$, by \eqref{distr_norm_trace} we have
$$
\int_\Omega \varphi \chi_r \,d\lambda + \int_\Omega \varphi u\cdot \nabla \chi_r  \,dy + \int_\Omega \chi_r u\cdot\nabla \varphi \,dy=0.
$$
Letting $r\rightarrow 0$,  since $\Omega$ is open, we have 
$$
\int_\Omega \chi_r u\cdot \nabla \varphi \,dy\rightarrow \int_\Omega u\cdot \nabla \varphi \,dy \qquad \text{and} \qquad \int_\Omega \varphi \chi_r \,d\lambda \rightarrow \int_\Omega \varphi \,d\lambda.
 $$
 Thus, by \cref{P:strong trace convergence} we conclude
 $$
\int_\Omega \varphi  \,d\lambda +  \int_\Omega  u\cdot\nabla \varphi \,dy=\int_{\partial \Omega} u^{\partial \Omega}_n \varphi \,d\mathcal H^{d-1}.
$$
According to \eqref{distr_norm_trace} the left hand side of the above equation equals $\int_{\partial \Omega} \tr_n(u;\partial\Omega) \varphi \,d\mathcal H^{d-1}$, which concludes the proof by the arbitrariness of $\varphi$.
\end{proof}

We are left to prove the key \cref{P:traces and tub neigh general}.

\begin{proof}[Proof of \cref{P:traces and tub neigh general}]
Let $\delta>0$ be fixed. By Lusin's theorem we find a closed set $A_1 \subset \Sigma$ such that $\mathcal{H}^{d-1}(\Sigma \setminus A_1) < \frac{\delta}{2}$ and $f^\Sigma$ is continuous on $A_1$. Let $r_k\rightarrow 0$ be any sequence. By Egorov's theorem, we find $A_2\subset \Sigma$, with $\mathcal{H}^{d-1}(\Sigma \setminus A_2) < \frac{\delta}{2}$, such that the convergence in \eqref{eq:trace general} is uniform on $A_2$. To sum up, by setting $A:=A_1\cap A_2 \subset \Sigma$ we have
\begin{equation}\label{A_is_small}
    \mathcal{H}^{d-1}(\Sigma \setminus A) < \delta,
\end{equation}
$f^\Sigma$ is continuous on $A$ and there exists $k_0\in \N$ such that for any $k>k_0$ and for any $x \in A$ it holds
\begin{equation} \label{eq: smallness 1}
    \frac{1}{r_k^d} \int_{\Omega \cap B_{5r_k}(x)} \abs{f(y) - f^{\Sigma}(x)}\, dy <\delta.
\end{equation}
By Tietze extension theorem we find a continuous function $\tilde f^\Sigma:\partial \Omega\rightarrow \R$ such that $\tilde f^\Sigma \equiv  f^\Sigma$ on $A$ and $\|\tilde f^\Sigma\|_{L^\infty(\partial \Omega)} \leq \norm{ f^\Sigma}_{L^\infty(\Sigma)}$. Since $\partial\Omega$ is compact, $\tilde f^\Sigma$ is uniformly continuous. Denote by $\Tilde{\gamma}$ its modulus of continuity. Let $\varphi \in C^\infty_c(\R^d)$ be any test function. Then, by using the projection onto $\partial \Omega$ defined in \cref{l:projection onto boundary}, we split the integral 
\begin{align}
    \abs{\frac{1}{r_k} \int_{(\Sigma)^{\rm in}_{r_k}} \varphi f \, dy - \int_{\Sigma} \varphi f^\Sigma \, d \mathcal{H}^{d-1} }  \leq & \norm{\varphi}_{L^\infty(\R^d)} \int_{\Sigma} \abs{ f^{\Sigma} - \tilde f^\Sigma } \, d \mathcal{H}^{d-1}
    \\ &  + \abs{\frac{1}{r_k} \int_{(\Sigma)^{\rm in}_{r_k}} \varphi \tilde f^\Sigma\circ \pi_{\partial \Omega} \, dy - \int_{\Sigma} \varphi \tilde f^\Sigma \, d \mathcal{H}^{d-1} }
    \\ &  + \frac{\norm{\varphi}_{L^\infty(\R^d)}}{r_k} \int_{ (\Sigma)^{\rm in}_{r_k} \setminus (A)_{r_k} } \abs{ f - \tilde f^\Sigma \circ \pi_{\partial \Omega} } \, d y 
    \\ &  + \frac{\norm{\varphi}_{L^\infty(\R^d)}}{r_k} \int_{(\Sigma)^{\rm in}_{r_k} \cap (A)_{r_k} } \abs{f - \tilde f^\Sigma\circ \pi_{\partial \Omega}} \, dy 
    \\ = & I_k + II_k + III_k + IV_k. 
\end{align}
By \eqref{A_is_small} together with \cref{R:trace is bounded} we have
\begin{equation}
    I_k \lesssim \mathcal{H}^{d-1}(\Sigma \setminus A) \norm{ f^\Sigma}_{L^\infty(\Sigma)} \lesssim \delta. 
\end{equation}
Moreover,  \cref{l:projection onto boundary}  implies that $\tilde f^\Sigma\circ \pi_{\partial \Omega}$ is continuous on $\partial \Omega$. Thus, by  \cref{C:conv one side mink} and \cref{p: almost continuous test function} we deduce $\lim_{k \to \infty} II_{k} = 0$.
Note that here we are allowed to apply \cref{p: almost continuous test function} since our sequence of measures is concentrated on compact sets, thus the two notions of convergence \eqref{tight_convergence} and \eqref{eq: weak convergence test} are equivalent.

To estimate $III_{k}$, since both $\Sigma$ and $A$ are closed $(d-1)$-rectifiable sets, by \cref{p: minkowski federer} it holds that 
\begin{equation}
    \lim_{r \to 0} \frac{\mathcal{H}^d ((\Sigma)_r) }{ \omega_{1}r} =  \mathcal{H}^{d-1}(\Sigma) \qquad \text{ and } \qquad \lim_{r \to 0} \frac{\mathcal{H}^d ((A)_r) }{\omega_{1}r} =  \mathcal{H}^{d-1}(A).
\end{equation}
Thus, we infer that 
\begin{equation}
    \lim_{r \to 0} \frac{\mathcal{H}^d ( (\Sigma)_r \setminus (A)_r ) }{\omega_{1}r} =  \mathcal{H}^{d-1}(\Sigma \setminus A) < \delta,
\end{equation}
from which we deduce 
\begin{align}
    \limsup_{k \to \infty} III_k \lesssim \limsup_{k \to \infty} \frac{\mathcal{H}^d \left((\Sigma)^{\rm in}_{r_k} \setminus (A)_{r_k}  \right) }{r_k}    \lesssim \limsup_{k \to \infty} \frac{\mathcal{H}^d ((\Sigma)_{r_k} \setminus (A)_{r_k}) }{r_k} \lesssim \delta. 
\end{align} 
We are left with $IV_k$. By Vitali's covering lemma we find a disjoint family of balls $\{B_r(x_j)\}_{j \in J}$ such that $x_j \in A$ for any $j \in J$ and 
$$ (A)_r \subset \bigcup_{j \in J} B_{5 r}(x_j). $$
Since the Minkowski dimension of $\partial \Omega$ is $d-1$, for sufficiently small radii $r$ it must hold that 
\begin{equation}
    \# J \lesssim r^{1-d}. \label{eq: cardinality of covering}  
\end{equation}
Then, recalling that $\tilde f^\Sigma=  f^\Sigma$ on $A\subset \Sigma$, we have that 
\begin{align}
    IV_k & \lesssim \frac{1}{r_k} \sum_{j \in J} \int_{\Omega \cap B_{5 r_k}(x_j)} \abs{f(y)  - \tilde f^\Sigma( \pi_{\partial \Omega}(y)) }\, dy 
    \\ & \leq \frac{1}{r_k} \sum_{j \in J} \int_{\Omega \cap B_{5 r_k}(x_j)} \abs{f (y) -  f^\Sigma (x_j)}\, dy + \frac{1}{r_k} \sum_{j \in J} \int_{\Omega \cap B_{5r_k}(x_j)} \abs{\tilde f^\Sigma(x_j) - \tilde f^\Sigma(\pi_{\partial\Omega}(y))}\, dy 
    \\ & = IV^1_k + IV^2_k. 
\end{align}
By \eqref{eq: smallness 1} and \eqref{eq: cardinality of covering}, for $k>k_0$ we infer that $IV^1_k \lesssim \# J \delta r_k^{d-1} \lesssim \delta$.
For any $j \in J$ and for almost every $y \in B_{5r_k}(x_j) \cap \Omega$, by the minimality of $\pi_{\partial \Omega} (y)$ we get
$$\abs{\pi_{\partial \Omega}(y) - x_j} \leq \abs{\pi_{\partial \Omega}(y) - y} + \abs{y - x_j} \leq 2 \abs{y-x_j} \leq 10 r_k, $$
from which, recalling that $\tilde \gamma$ is the modulus of continuity of $\tilde f^\Sigma$, we deduce
\begin{align}
    IV_k^2 \lesssim \frac{\# J}{r_k} \Tilde{\gamma}(10 r_k) \mathcal{H}^d(B_{5r_k}) \lesssim \Tilde{\gamma}(10 r_k).
\end{align}
Thus, we achieved
\begin{equation}
    \limsup_{k \to \infty} IV_k \lesssim \lim_{k \to \infty} \Tilde{\gamma}(10 r_k) + \delta \lesssim \delta. 
\end{equation}
To summarize, we have shown that 
\begin{equation}
    \limsup_{k \to \infty} \abs{ \frac{1}{r_k} \int_{ (\Sigma)^{\rm in}_{r_k}} \varphi f \, dy - \int_{\Sigma} \varphi  f^\Sigma \, d \mathcal{H}^{d-1} } \lesssim \delta.
\end{equation}
The conclusion immediately follows since both $r_k\rightarrow 0$ and $\delta>0$ are arbitrary.
\end{proof}

Since it might be of independent interest, we also state the following result, which in turn generalizes \cite{DRINV23}*{Proposition 5.3}.
\begin{corollary}\label{P:strong trace convergence abs value}
Let $u\in L^\infty(\Omega)$ be a vector field, $\Omega\subset \R^d$ a bounded open set with Lipschitz boundary and let $\chi_r$ be as in \eqref{eq:cut_off}.
Assume that $u$ has an outward normal Lebesgue trace $u^{\partial \Omega}_n$ on $\partial \Omega$ according to \cref{D:Leb normal trace}.
Then, for any $\varphi\in C^\infty_c(\R^d)$, it holds
\begin{equation}
\lim_{r\rightarrow 0} \int_\Omega \varphi \left|u\cdot \nabla \chi_r \right| \,dy=\int_{\partial \Omega} 
 \varphi \left|u^{\partial \Omega}_n\right| \,d\mathcal H^{d-1}.
 \end{equation}
\end{corollary}
Notice that 
\[
    \frac{1}{r^d}\int_{B_r(x)\cap\Omega} \abs{|u\cdot\nabla d_{\partial\Omega}|-|u_{-n}^{\partial\Omega}(x)|} \, dy \leq
    \frac{1}{r^d}\int_{B_r(x)\cap\Omega} \left|u\cdot\nabla d_{\partial\Omega}-u_{-n}^{\partial\Omega}(x)\right| \, dy.
\]
Thus, \cref{P:strong trace convergence abs value} follows again by \cref{P:traces and tub neigh general}. By recalling the notion of traces for $BV$ functions from \cref{T:trace_in_BV}, we also get the following corollary, which will be useful later on in \cref{S:Cont Eq}.
\begin{corollary}\label{p: restriction of traces sobolev}
Let $\Omega \subset \R^d$ be a bounded open set with Lipschitz boundary and let $f \in BV(\Omega)\cap L^\infty (\Omega)$. Let $\Sigma \subset \partial \Omega$ be a closed set. Then, it holds
\begin{equation}
    \lim_{r \to 0} \frac{1}{r} \int_{(\Sigma)_{r}^{\rm in}} f \, dx = \int_{\Sigma} f^{\Omega} \, d \mathcal{H}^{d-1}. 
\end{equation}
\end{corollary}

\subsection{Positive and negative normal Lebesgue traces}\label{S:NLT local and signed}

We start by specifying how \cref{D:Leb normal trace} trivially extends to the case in which only a portion of the boundary $\Sigma\subset \partial \Omega$ is considered.

\begin{definition}\label{D:NLT local}
    Let $\Omega\subset \R^d$ be a bounded Lipschitz open set, $\Sigma\subset \partial \Omega$ measurable and $u\in L^1(\Omega)$ a vector field. We say that $u$ admits an inward Lebesgue normal trace on $\Sigma$ if there exists a function  $f\in L^1(\Sigma;\mathcal H^{d-1})$ such that, for every sequence $r_k\rightarrow  0$, it holds
    $$
    \lim_{k\rightarrow \infty}\frac{1}{r_k^d} \int_{B_{r_k} (x)\cap\Omega} \left|(u\cdot \nabla d_{\partial \Omega})(y)- f(x)\right|\,dy=0\qquad \text{for } \mathcal{H}^{d-1}\text{-a.e. } x\in \Sigma.
    $$
Whenever such a function exits, we will denote it by $f=: u_{-n}^{\Sigma}$. Consequently, the outward Lebesgue normal trace on $\Sigma$ will be $u_{n}^{\Sigma}:= - u_{-n}^{\Sigma}$.
\end{definition}
Even if the definition is given on any $\Sigma\subset \partial \Omega$, the meaningful case is $\mathcal H^{d-1}(\Sigma)>0$.

\begin{remark}
The same definition can be given for any oriented Lipschitz hypersurfaces $\Sigma\subset \R^d$. In the case $\Sigma \subset \partial\Omega$, an orientation is canonically induced on $\Sigma$. Since it will be sufficient for our purposes, we will only deal with such a case.
\end{remark}

It is rather easy to show that the positive and negative part of the normal Lebesgue trace behave well as soon as the latter exists. Indeed, if $f_+$ and $f_-$ are the positive and the negative part of a function $f$, that is $f=f_+-f_-$, from $|f_--g_-|,|f_+-g_+|\leq |f-g|$ we immediately obtain the following result.

\begin{proposition} \label{P:NLT pos and neg}
Let $\Omega\subset \R^d$ be a bounded Lipschitz open set, $\Sigma\subset \partial \Omega$ be measurable, $u\in L^1(\Omega)$ a vector field which has a normal Lebesgue trace $u^\Sigma_n$ according to \cref{D:NLT local}. Then, for every sequence $r_k\rightarrow  0$, we have
$$ \lim_{k\rightarrow \infty}\frac{1}{r_k^d} \int_{B_{r_k} (x)\cap\Omega} \left|(u\cdot \nabla d_{\partial \Omega})_+(y)- \left(u^\Sigma_{-n}\right)_+(x)\right|\,dy=0\qquad \text{for } \mathcal{H}^{d-1}\text{-a.e. } x\in \Sigma $$
and 
$$ \lim_{k\rightarrow \infty}\frac{1}{r_k^d} \int_{B_{r_k} (x)\cap\Omega} \left|(u\cdot \nabla d_{\partial \Omega})_-(y)- \left(u^\Sigma_{-n}\right)_-(x)\right|\,dy=0\qquad \text{for } \mathcal{H}^{d-1}\text{-a.e. } x\in \Sigma. $$
\end{proposition}

As an almost direct consequence we have the following result. 

\begin{corollary} \label{C:NLT pos and neg}
Let $\Omega\subset \R^d$ be a bounded open set with Lipschitz boundary, $\Sigma\subset \partial \Omega$ be measurable, $u\in L^  \infty(\Omega)$ a vector field which has an outward normal Lebesgue trace $u^\Sigma_n$. The following facts are true.
\begin{itemize}
     \item[(i)] If $\tilde \Sigma \subset \Sigma$ is any closed set on which $u^\Sigma_{n}\big|_{\tilde \Sigma} \geq 0$, we have 
\begin{equation}\label{eq:pos norm trace exit}
\lim_{r\rightarrow 0}\frac{1}{r} \int_{(\tilde \Sigma)^{\rm in }_r} (u\cdot \nabla d_{\partial \Omega})_+(y)\,dy=0.
 \end{equation}
\item[(ii)] If $\tilde \Sigma \subset \Sigma$ is any closed set on which $u^\Sigma_{n}\big|_{\tilde \Sigma} \leq 0$, we have 
         $$
\lim_{r\rightarrow 0}\frac{1}{r} \int_{(\tilde \Sigma)^{\rm in }_r} (u\cdot \nabla d_{\partial \Omega})_-(y)\,dy=0.
$$
     \end{itemize}
\end{corollary}
Restricting to closed subsets of $\Sigma$ in the above result is necessary. Even if $u^\Sigma_n$ has distinguished sign on $\Sigma$, we can not expect the conclusions of \cref{C:NLT pos and neg} to hold replacing $\tilde \Sigma$ by $\Sigma$. Indeed, $\Sigma$ could be countable and dense in $\partial \Omega$, from which $(\Sigma)_r=(\partial \Omega)_r$ for any $r>0$,  but it is clear that any assumption on a $\mathcal H^{d-1}$-negligible subset of $\partial\Omega$ will not suffice to deduce anything in the whole $(\partial \Omega)_r$.

\begin{proof}
    Since $u^\Sigma_{-n}\big|_{\tilde \Sigma} =-u^\Sigma_{n}\big|_{\tilde \Sigma}\leq  0$, by \cref{P:NLT pos and neg} we deduce 
    $$
     \lim_{k\rightarrow \infty}\frac{1}{r_k^d} \int_{B_{r_k} (x)\cap\Omega} (u\cdot \nabla d_{\partial \Omega})_+(y)\,dy=0\qquad \text{for } \mathcal{H}^{d-1}\text{-a.e. } x\in \tilde \Sigma.
    $$
Then, \eqref{eq:pos norm trace exit} follows by applying \cref{P:traces and tub neigh general} with $f = (u \cdot \nabla d_{\partial \Omega})_+$, $f^{\tilde{\Sigma}} = 0$ and $\varphi=1$. The proof of $(ii)$ is completely analogous. 
\end{proof}

We conclude this section by showing that $BV$ vector fields satisfy \eqref{eq:pos norm trace exit} with $\tilde \Sigma$ being the portion of the boundary where $u$ is pointing outward. In particular, this shows that our assumption \eqref{hp: gamma out exit} automatically holds if the vector field $u$ is $BV$ up to the boundary. Thus \cref{T:CE general} offers an honest generalization of \cite{CDS14}. In some sense, and as expected, the next proposition shows that $BV$ vector fields achieve the positive and negative values of their normal trace in a uniform integral sense.

\begin{proposition} \label{P: BV and exit}
Let $\Omega\subset \R^d$ be a bounded open set with Lipschitz boundary and let $u\in BV(\Omega)\cap L^\infty(\Omega)\subset \mathcal{MD}^\infty(\Omega)$. We set 
$$ \Sigma^+:=\left\{x\in \partial \Omega \,:\, \tr_n(u;\partial \Omega)(x)\geq 0\right\} \qquad \text{and} \qquad \Sigma^-:=\partial\Omega\setminus \Sigma^+. $$
Then
\begin{equation}\label{bv out unif}
\lim_{r\rightarrow 0}\frac{1}{r} \int_{(\tilde \Sigma)^{\rm in }_r} (u\cdot \nabla d_{\partial \Omega})_+(y)\,dy=0 \qquad \forall \tilde \Sigma \subset \Sigma^+ \text{ closed}
\end{equation}
and 
\begin{equation}\label{bv in unif}
\lim_{r\rightarrow 0}\frac{1}{r} \int_{(\tilde \Sigma)^{\rm in }_r} (u\cdot \nabla d_{\partial \Omega})_-(y)\,dy=0 \qquad \forall \tilde \Sigma \subset \Sigma^- \text{ closed}.
\end{equation}
\end{proposition}
\begin{proof}
    Since $u\in BV(\Omega)$, by \cite{DRINV23}*{Proposition 5.5} we deduce that $u$ admits a normal Lebesgue trace $u^{\partial\Omega}_n$ in the sense of \cref{D:Leb normal trace}. Moreover, \cref{T:strong trace equals weak} implies that $u^{\partial\Omega}_n\big|_{\Sigma^+}\geq 0$, from which \eqref{bv out unif} directly follows by applying \cref{C:NLT pos and neg}. The proof of \eqref{bv in unif} is completely analogous.
\end{proof}
The global $BV(\Omega)$ assumption could have been relaxed to hold only locally around $\Sigma^+$ and $\Sigma^-$, possibly also up to a negligible subset of the boundary. 

\subsection{Lebesgue traces are strictly stronger than distributional traces}\label{S:couter ex} 

In this section we provide an example of a $2$-dimensional bounded divergence-free vector field which has zero normal distributional trace, but does not admit a normal Lebesgue trace. Denote by $\{e_1,e_2\}$ the canonical orthonormal basis of $\R^2$. Consider the square $Q:=\{(x,y)\in\R^2 : 0\leq x<1, \ 1\leq y<2 \}$ and its rescaled and translated copies
\[
    Q_{i,j}:=2^{-j}(Q+i e_1)=\left\{(x,y)\in\R^2\, \colon \,\frac{i}{2^j}\leq x<\frac{i+1}{2^j}, \ 2^{-j}\leq y<2^{-j+1}  \right\}, \qquad \forall i,j\in\Z.
\]
Notice that this family of squares tiles the upper half-plane $\R^2_+$ (see \cref{grid_figure}).

\begin{figure}
\includegraphics[width=0.6\textwidth]{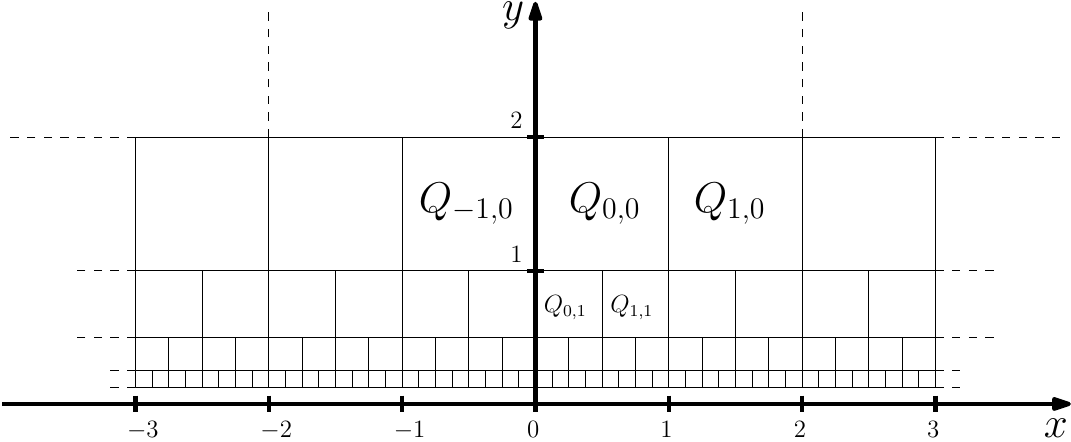}
\caption{The \quotes{tiles} $Q_{i,j}$ become finer approaching the axis $\{y=0\}$.}
\label{grid_figure}
\end{figure}

Take any (possibly smooth) bounded divergence-free vector field $v:Q\to \R^2$ tangent to $\partial Q$, define the corresponding rescaled vector fields $v_{i,j}: Q_{i,j} \to \R^2$ as
\[
    v_{i,j}(x,y):= v(2^jx-i,2^jy)
\]
and combine them to get $u:\R^2_+\to\R^2$ by setting $u=v_{i,j}$ on each $Q_{i,j}$. We establish some properties of this vector field. In the lemma below we will denote by $\tr_n (u;\partial\Omega)$ the distributional normal trace on $\partial \Omega$ according to \cref{D:ditrib norm trace}.

\begin{lemma} \label{l: example}
The vector field $u:\R^2_+\to\R^2$ defined above satisfies the following properties:
\begin{itemize}
    \item[(i)] $u$ is bounded, divergence-free and $\tr_n(u;\partial \R^2_+)\equiv 0$;
    \item[(ii)] the restriction of $u$ to the strip $\{0<y<2^{-k}\}$ is $2^{-k-1}$-periodic in the first variable;
    \item[(iii)] for any $(\bar{x},0) \in \partial \R^2_+$ it holds that 
    \begin{equation}
        \liminf_{r \to 0} \frac{1}{r^2} \int_{B_r^+((\bar{x},0))} \abs{u \cdot e_2} \, dx \, dy \geq \frac{1}{8} \int_{Q} \abs{v \cdot e_2}\, dx \, dy. \label{eq: liminf example}
    \end{equation}
\end{itemize}
In particular, if $v$ satisfies
        \begin{equation}\label{cond strict positive}
            \int_Q |v\cdot e_2| \, dx \, dy >0,
        \end{equation}
        then $u$ does not admit a normal Lebesgue trace on $\partial\R^2_+=\R$ in the sense of \cref{D:Leb normal trace}.
\end{lemma}

\begin{proof} 
Since $v$ is bounded, the same holds for $u$. Given $\varphi\in C^1_c(\R^2)$ we compute
        \[
            \int_{\R^2_+} u(x,y)\cdot \nabla\varphi(x,y) \, dx \, dy 
            = \sum_{i,j\in\Z} \int_{Q_{i,j}} v_{i,j}\cdot \nabla\varphi \, dx \, dy 
        \]
        and since $v_{i,j}$ is divergence-free
        \[
             \int_{\R^2_+} u(x,y)\cdot \nabla\varphi(x,y) \, dx \, dy
             =  \sum_{i,j\in\Z} \int_{\partial Q_{i,j}} \varphi \, v_{i,j}\cdot n_{Q_{i,j}} \, d \mathcal{H}^1 =0,
        \]
        where in the last equality we have used that $v_{i,j}$ is tangent to $\partial Q_{i,j}$ since $v$ is tangent to $\partial Q$. This computation shows that $u$ is divergence-free (for this it would have been enough to test with $\varphi\in C^1_c(\R^2_+)$) and that the normal distributional trace vanishes according to \cref{D:ditrib norm trace}, thus proving $(i)$.

To check $(ii)$ it is enough to prove that the restriction of $u$ to the strip $\{2^{-k-1}\leq y<2^{-k}\}$ is $2^{-k-1}$-periodic in the first variable. This is evident, since a point $(x,y)$ in this strip belongs to some square $Q_{i,k+1}$ and hence
        \begin{align}
            u(x+2^{-k-1},y)&=v_{i+1,k+1}(x+2^{-k-1},y)\\
            &=v(2^{k+1}x-1-i+1,2^{k+1}y) \\
            &=v(2^{k+1}x-i,2^{k+1}y)\\
            &=v_{i,k+1}(x,y)\\
            &=u(x,y).
        \end{align}

We are left with $(iii)$. Notice that in this case $\nabla d_{\partial \R^2_+} (x,y)=e_2$ for any $(x,y)\in \R^2_+$. Hence, given $(\bar{x},0)\in \partial\R^2_+$, for $2^{-k}\leq r < 2^{-k+1}$ one can estimate
 \begin{align}
    \frac{1}{r^2} \int_{B^+_r((\bar{x},0))} |u\cdot e_2| \, dx \, dy &\geq 
    4^{k-1} \int_{B^+_{2^{-k}}((\bar{x},0))} |u\cdot e_2| \, dx \, dy\\
    &\geq 4^{k-1}  \int_{[\bar{x}-2^{-k-1},\bar{x}+2^{-k-1}]\times [0,2^{-k-1}]} |u\cdot e_2| \, dx \, dy.
  \end{align}
By exploiting the periodicity with respect to the first variable ($u$ is $2^{-k-2}$-periodic in $x$ in the strip $0<y<2^{-k-1}$), we get
\begin{align}
    4^{k-1}  \int_{[\bar{x}-2^{-k-1},\bar{x}+2^{-k-1}]\times [0,2^{-k-1}]} \abs{u\cdot e_2} \, dx \, dy &= 4^{k-1} \int_{[0,2^{-k}]\times[0,2^{-k-1}]} \abs{u\cdot e_2} \, dx \, dy \\
    &= 4^{k-1} \sum_{j\geq k+2} \sum_{i=0}^{2^{j-k}-1} \int_{Q_{i,j}} \abs{ v_{i,j}\cdot e_2} \, dx\, dy \\
    &= 4^{k-1} \sum_{j\geq k+2} \sum_{i=0}^{2^{j-k}-1} 4^{-j} \int_Q \abs{ v\cdot e_2} \, dx \, dy \\
    &= \int_Q \abs{v\cdot e_2} \, dx \, dy \sum_{j\geq k+2} 2^{k-j-2}\\
    &= \frac{1}{8} \int_Q \abs{v\cdot e_2} \, dx \, dy. 
\end{align}

thus proving \eqref{eq: liminf example}. To conclude, by \cref{T:strong trace equals weak}, we know that if $u$ admits a normal Lebesgue trace, then it has to vanish. In particular, if $v$ satisfies \eqref{cond strict positive}, by \eqref{eq: liminf example} we infer that $u$ does not admit a normal Lebesgue trace on $\partial \R^2_+$.
\end{proof}

\begin{remark}
For completeness we give an explicit example of a vector field $v$ satisfying the assumptions of \cref{l: example} and \eqref{cond strict positive}. Let us define 
\[
    v(x,y)=\Big(\sin(2\pi x)\cos(2\pi y) , - \sin(2\pi y)\cos(2\pi x)\Big).
\]
It is apparent that $v$ is divergence-free, tangent to $\partial Q$ and
\[
    \int_Q |v\cdot e_2| \, dx \,  dy = \int_0^1\int_1^2 |\sin(2\pi y)\cos(2\pi x)| \, dx \, dy = \frac{4}{\pi^2}>0.
\]
\end{remark}

\section{On the continuity equation on bounded domains}\label{S:Cont Eq}

We now give the precise definition of weak solutions to the continuity equation \eqref{CE}. We follow \cites{CDS14,CDS14bis}. For a regular solution $\rho$ to \eqref{CE}, testing the equation with $\varphi \in C^1(\overline{\Omega} \times [0,T))$, we get
\begin{equation}
\int_0^T \int_\Omega \rho (\partial_t \varphi + u \cdot \nabla \varphi) + (c\rho+f)\varphi \, dx\, dt  = \int_0^T \int_{\partial\Omega}\rho u \cdot n_\Omega \varphi \, d\mathcal{H}^{d-1} \,dt  -\int_\Omega \rho(x,0)\varphi(x,0)\,dx . \label{eq: int formulation CE boundary}
\end{equation}
We aim to give meaning to the integral formulation \eqref{eq: int formulation CE boundary} for rough solutions. To prescribe the boundary condition of $\rho$ on $\partial\Omega$, we exploit the theory of normal distributional traces. 
Let $u, \rho, c, f$ as in \cref{T:CE general} and let $\rho \in L^\infty(\Omega \times (0,T))$ be an \quotes{interior in $\Omega$} distributional solution to 
\begin{equation}\label{eq: CE without boundary}
    \left\{
    \begin{array}{rcll}
        \partial_t \rho + \div (\rho u ) &=& c \rho +f & \text{ in }\Omega\times (0,T) \\
        \rho(\cdot,0) &=&  \rho_0 & \text{ in }  \Omega,
    \end{array}
    \right.
\end{equation}
that is, we restrict \eqref{eq: int formulation CE boundary} to test functions in $C^\infty_c([0,T)\times \Omega)$. Thus, for the moment, we are not interested in the boundary datum on $\partial \Omega$. Moreover, setting $U=(u,1)$, we notice that the space-time divergence of $\rho U$ satisfies 
\begin{equation}\label{E:CE}
    \Div(\rho U) = \div(\rho u) + \partial_t \rho = c\rho+f,
\end{equation}
that is $\rho U \in \mathcal{MD}^\infty(\Omega\times (0,T))$ and, recalling \eqref{distr_norm_trace}, we can define
\[
    \tr_n(\rho U ; \partial(\Omega\times(0,T))) \in L^\infty(\partial(\Omega\times(0,T))).
\]

We remark that the latter is a space-time normal trace. Indeed, for any fixed time $t\in(0,T)$, the vector field $\rho (\cdot,t) u(\cdot,t)$ might fail to belong to $\mathcal{MD}^\infty(\Omega)$. Thus, we need to consider the restriction of the space-time trace to $\Lambda:=\partial\Omega \times (0,T)$, that is
\begin{equation}
    \tr_n(\rho u) := \tr_n(\rho U; \partial(\Omega\times(0,T)))\big|_{\Lambda}. \label{eq: tr rho u}
\end{equation}
Similarly, we define
\begin{equation}
    \tr_n(u) := \tr_n(U; \partial(\Omega\times(0,T)))\big|_{\Lambda}. \label{eq: tr u}
\end{equation}
In this case, since $ u(\cdot, t) \in \mathcal{MD}^\infty(\Omega)$ for a.e. $t$ (recall that we are assuming $\div u \in L^1(\Omega\times (0,T))$), the normal trace at fixed time of $u$ is well defined and it can be checked that $\tr_n(u(\cdot, t ); \partial \Omega) = \tr_n(u)(\cdot, t)$ for a.e. $t$ as functions in $L^\infty(\partial \Omega)$ (see \cref{l: measurability of space-time trace}). Since the outer normal to the set $\Omega\times(0,T)$ is given by the vector $(n_\Omega,0)$ at any point of $\Lambda$, then the traces $\tr_n(u)$ and $\tr_n(\rho u)$ describe the components of $u$ and $\rho u$ (respectively) which are normal to $\partial\Omega$. 

\begin{definition} \label{D:Gamma plus and minus} 
With the above notation, we define the space-time set where $u$ is not entering the domain $\Omega$ by 
\[
    \widetilde{\Gamma}^+:=\left\{(x,t)\in\Lambda\, :\, \tr_n(u)(x,t)\geq 0\right\}.
\] 
We define $\Gamma^+ := \spt(\H^d \llcorner \widetilde{\Gamma}^+)$.
Then, the set where $u$ is entering $\Omega$ is defined by $\Gamma^-:=\Lambda\setminus\Gamma^+ \subset \Lambda$.
\end{definition}

\begin{remark}
In \cref{D:Gamma plus and minus}, the set $\widetilde{\Gamma}^+$ is defined up to $\H^d$-negligible sets. Then, we replace it with the closed (in $\Lambda$) representative given by $\Gamma^+$ for convenience. Hence, the set $\Gamma^-$ is open in $\Lambda$. We remark that $\Gamma^+$ is independent of the choice of the representative of $\tr_n(u)$. In view of the assumption $(i)$ in \cref{T:CE general}, if we set $\Gamma^{\pm }$ as in \cref{D:Gamma plus and minus}, $\Gamma^-$ is the smallest subset of $\Lambda$ around which additional regularity of the vector field is assumed. 
\end{remark}

We notice that in the regular setting, we would have
\[
    \rho=\frac{(\rho u)\cdot n_\Omega}{u\cdot n_\Omega} = \frac{\tr_n(\rho u)}{\tr_n(u)},
\]
provided $u\cdot n\neq 0$. Then, the idea is to impose the validity of the above formula to make sense of the boundary datum in the rough setting. Motivated by the discussion above, we recall the definition of weak solution to the initial-boundary value problem \eqref{CE} from \cites{CDS14,CDS14bis}. 

\begin{definition} \label{D:Cont Eq weak sol} 
Given $u,c,f,\rho_0,g$ as in \cref{T:CE general}, we say that $\rho$ solves the continuity equation \eqref{CE} with boundary condition $\rho=g$ on $\Gamma^-$ and initial datum $\rho_0$, if for any $\varphi\in C^1_c\left(\overline{\Omega}\times[0,T)\right)$
\[
    \int_0^T \int_\Omega \rho (\partial_t \varphi + u \cdot \nabla \varphi) + (c\rho+f)\varphi \, dx\, dt  = \int_0^T \int_{\partial\Omega}\tr_n(\rho u) \varphi \, d\mathcal{H}^{d-1} \,dt -\int_\Omega \rho_0(x)\,\varphi(x,0)\,dx 
    \]
and $\tr_n(\rho u)=g \tr_n(u)$ on $\Gamma^-$, where $\tr_n(\rho u), \tr_n(u)$ and $\Gamma^-$ are defined by \eqref{eq: tr rho u}, \eqref{eq: tr u} and \cref{D:Gamma plus and minus}.
\end{definition}

By the interior renormalization property established in \cite{Ambr04}, under the assumptions of \cref{T:CE general}, for any choice of $\beta\in C^1(\R)$, in the above notation, we have that 
\begin{equation} \label{eq: renormalized equation}
    \Div(\beta(\rho) U) = \div(\beta(\rho) u) + \partial_t (\beta(\rho)) = c\rho\beta'(\rho)+f\beta'(\rho)+(\beta(\rho)-\rho\beta'(\rho))\div u,
\end{equation} 
in $\mathcal{D}'( \Omega\times (0,T))$. Hence,  $\beta(\rho) U\in \mathcal{MD}^\infty (\Omega\times(0,T))$ and we set
\begin{equation} \label{eq: tr beta(rho) u}
    \tr_n(\beta(\rho) u) := \tr_n(\beta(\rho) U;\partial(\Omega\times(0,T)))\big|_\Lambda
    \in L^\infty (\Lambda).
\end{equation}
We prove a chain rule for the distributional normal trace of weak solutions to \eqref{CE}. The proof follows closely that of \cite{ACM05}*{Theorem 4.2}  and \cite{BB20}*{Proposition 5.10}  and it relies on the Gagliardo extension theorem \cite{Ga57} (see also \cite{Le17}*{Theorem 18.15} for a modern reference).  

\begin{proposition} \label{P:chain-rule} 
Under the assumptions of \cref{T:CE general}, for any $\beta \in C^1(\R)$ let  $\tr_n(\rho u)$, $\tr_n(u)$ and $\tr_n(\beta(\rho) u)$ be defined by \eqref{eq: tr rho u}, \eqref{eq: tr u} and \eqref{eq: tr beta(rho) u} respectively. Then, it holds that
\begin{equation}
\tr_n(\beta(\rho) u) =  \beta\left(\frac{\tr_n(\rho u)}{\tr_n(u)}\right) \tr_n(u)  \qquad \text{ in } O\cap\Lambda, \label{eq: chain rule formula}
\end{equation}
where the term $\beta\left(\frac{\tr_n(\rho u)}{\tr_n(u)}\right)$ is arbitrarily defined in the set where $\tr_n(u) = 0$. 
\end{proposition}

\begin{proof} 
For the sake of clarity, we split the proof in several steps. 

\textsc{\underline{Step 0}: Assume $W^{1,1}$ regularity.} First, assume in addition that both $\rho$ and $u$ enjoy $W^{1,1}(\Omega\times(0,T))$ regularity. In this case, $\beta(\rho) \in W^{1,1}(\Omega \times (0,T))$ and the normal traces of $u,\rho u, \beta(\rho) u$ can be computed explicitly by \cref{l: product rule}. Thus, \eqref{eq: chain rule formula} follows immediately. In the rest of the proof, we show how to extend the vector fields $u$ and $\rho u$ from $\Omega\times(0,T)$ to $\R^d\times(0,T)$ in such a way that all the traces on the set $\Lambda\cap O$ from inside and outside coincide. Thus, we can compute the inner traces relying on the (stronger) outer ones.  

\textsc{\underline{Step 1}: Extension of $U$.}
By \cref{l: measurability of space-time trace}, we find a measurable function $\tr(u) \in L^\infty(\Lambda \cap O;\R^{d})$ such that $\tr(u)(\cdot, t) = u_t^\Omega$ as elements in $L^\infty(\partial \Omega \cap O_t)$, for a.e. $t \in (0,T)$. Here, $u_t^{\Omega}$ denotes the full $BV$ trace of $u(\cdot, t)$ on $\partial \Omega \cap O_t$. Then, with a slight abuse of notation, we define $U^{\Omega \times (0,T)} \in L^\infty(\Lambda;\R^{d+1})$ such that 
\begin{equation}
    U^{\Omega \times (0,T)}(x,t) = \begin{cases}
        (u_t^\Omega(x), 1) & \text{ if } x \in \partial \Omega \cap O_t
        \\ (0,1) & \text{ if } x \in \partial \Omega \cap (O_t)^c. 
    \end{cases}
\end{equation}
Since the trace operator is surjective from $W^{1,1}$ to $L^1$, by Gagliardo's theorem, we find a vector field $V\in W^{1,1}(\Omega^c\times(0,T);\R^{d+1})$ such that $V^{\Omega^c\times(0,T)}=U^{\Omega\times(0,T)}$. Since $U^{\Omega \times (0,T)} \in L^\infty(\Lambda)$, by a truncation argument we also have $V \in L^\infty(\Omega^c \times (0,T))$. Moreover, by the property of the Sobolev trace (see \cref{T:trace_in_BV}), it is immediate to see that $V$ can be taken of the form $V = (v,1)$. Thus, we define the extension of $U$ as
\begin{equation}
\tilde{U}(x,t)=
\begin{cases}
    U(x,t)& \text{if } (x,t)\in \Omega\times(0,T) \\
                     V(x,t)& \text{if } (x,t)\in \Omega^c\times(0,T).
                  \end{cases}
\end{equation}
We have that $\tilde U(x,t) = (\tilde u_t(x) ,1 )$, where 
\begin{equation}
    \tilde u_t(x) = \begin{cases}
        u_t(x) & \text{ if } x \in \Omega, 
        \\ v_t(x) & \text{ if } x \in \Omega^c. 
    \end{cases}
\end{equation}
Then, for almost every $t\in(0,T)$, by \cite{AFP00}*{Theorem 3.84} we have $\tilde{u}_t\in BV_{\rm loc}(O_t)$ and
\begin{equation}
    \abs{ \nabla \tilde{u}_t } (\partial\Omega\cap O_t) = \abs{ u_t^\Omega-v_t^{\Omega^c} }   \H^{d-1} (\partial\Omega\cap O_t) = 0. \label{E:zero mass boundary}
\end{equation}

\textsc{\underline{Step 2}: Extension of $\rho$.} Consider the function defined on $\Lambda$ by
\begin{equation}
    \theta = \frac{\tr_n(\rho u)}{\tr_n(u)} \mathbbm{1}_{\tr_n(u)\neq 0}.
\end{equation}
With the same argument as in the proof of \cite{ACM05}*{Theorem 4.2}, we have 
$$
\norm{\theta}_{L^\infty(\Lambda)}\leq \norm{\rho}_{L^\infty(\Omega\times (0,T))}.
$$ 
Therefore, again by Gagliardo's extension theorem, we find $\sigma\in W^{1,1}\cap L^\infty(\Omega^c\times(0,T))$ such that $\sigma^{\Omega^c\times(0,T)}=\theta$. Since $\sigma$ and  $V$ are space-time Sobolev, by \eqref{eq: product rule trace} and \cref{R: BV vs normal trace}, we have that 
\begin{equation}
    \tr_n(\sigma V, \partial(\Omega^c\times(0,T))) = \sigma^{\Omega^c\times(0,T)} V^{\Omega^c\times(0,T)} \cdot n_{\Omega^c} = -\theta \tr_n(u) = -\tr_n(\rho u) \quad \text{ on } \Lambda. \label{eq: bilateral trace sigma V}
\end{equation}
Then, setting
\begin{equation}
\tilde{\rho}(x,t)=
\begin{cases}
    \rho(x,t)& \text{if } (x,t)\in \Omega\times(0,T) \\
    \sigma(x,t)& \text{if } (x,t)\in \Omega^c\times(0,T)
\end{cases}
\end{equation}
and noticing that $\tilde \rho \tilde U \in \mathcal{MD}^\infty(\Omega \times (0,T)) \cup \mathcal{MD}^\infty(\Omega^c \times (0,T))$, by \cref{l: gluing lemma} we have that $\tilde \rho \tilde U \in \mathcal{MD}^\infty(\R^d\times (0,T))$ and by \eqref{eq: formula gluing} and \eqref{eq: bilateral trace sigma V} it holds
\begin{equation} \label{E:no mass 2}
    \abs{\Div(\tilde{\rho}\tilde{U})}(\Lambda)=0.
\end{equation}

\textsc{\underline{Step 3}: Proof of $\beta(\tilde{\rho})\tilde{U} \in \mathcal{MD}^\infty(\R^d\times (0,T))$.} By Ambrosio's renormalization theorem for $BV$ vector fields \cite{Ambr04}, we have that \eqref{eq: renormalized equation} holds in $\mathcal{D}'(\Omega \times(0,T))$. Hence, it is clear that $\beta(\tilde \rho) \tilde U \in \mathcal{MD}^\infty(\Omega \times (0,T))$. Moreover, since $\tilde \rho, \tilde U \in W^{1,1} \cap L^\infty (\Omega^c \times (0,T))$, we get $\beta(\tilde \rho) \tilde U \in \mathcal{MD}^\infty(\Omega^c \times (0,T))$ as well. Thus, by \cref{l: gluing lemma}, we infer that $\beta(\tilde \rho) \tilde U \in \mathcal{MD}^\infty(\R^d \times (0,T))$. We claim that 
\begin{equation} \label{eq: div renormalized does not weight lambda}
    \abs{\Div(\beta(\tilde{\rho})\tilde{U})} (\Lambda\cap O) = 0. 
\end{equation}
Given a space mollifier $\eta\in C^\infty_c(\R^d)$, denote by $\phi_\eps=\phi*\eta_\eps$ the space regularization of a function $\phi$. Then, since $\tilde U \in \mathcal{MD}^\infty( \R^d \times (0,T))$ (see \cref{l: gluing lemma}) and $\tilde \rho_\e$ is smooth in the spatial variable, we compute
\begin{align}
\Div (\beta(\tilde{\rho}_\eps) \tilde{U}) & = \beta(\tilde{\rho}_\eps) \Div \tilde{U} + \tilde{U} \cdot \nabla_{(x,t)} \beta(\tilde{\rho}_\e) \\
&= \beta(\tilde{\rho}_\e) \div \tilde{u} + \tilde{u} \cdot \nabla \beta(\tilde{\rho}_\e) + \partial_t \beta(\tilde{\rho}_\e)
\\ & = \beta(\tilde{\rho}_\e) \div \tilde{u} + \beta'(\tilde{\rho}_\e) \tilde{u} \cdot \nabla \tilde{\rho}_\e + \partial_t \beta(\tilde{\rho}_\e),
\end{align}
where the above identities hold in $\mathcal{D}'(\R^d \times (0,T))$. Since $\tilde U = (\tilde u, 1)$, by \eqref{E:CE}, \cref{l: gluing lemma} and \eqref{E:no mass 2} we get 
\begin{align}
    \partial_t \tilde{\rho}_\e & = \Div(\tilde{\rho}\tilde{U})_\e - \div(\tilde{\rho}\tilde{u})_\e \\
    &=\left( \mathbbm{1}_{\Omega\times(0,T)}\Div(\rho U)+\mathbbm{1}_{\Omega^c\times(0,T)}\Div(\sigma V)\right)_\e - \div(\tilde{\rho}\tilde{u})_\e 
    \\ & = \left( \mathbbm{1}_{\Omega\times(0,T)}(c\rho+f)+\mathbbm{1}_{\Omega^c\times(0,T)}(\sigma\Div V + V \cdot \nabla_{(x,t)}\sigma)\right)_\e - \div(\tilde{\rho}\tilde{u})_\e.
\end{align}
Therefore, we have that $\partial_t \tilde{\rho}_\e \in L^1(\R^d \times (0,T))$ and, by the chain rule for Sobolev functions, we deduce $\partial_t\beta(\tilde{\rho}_\e)=\beta'(\tilde{\rho}_\e)\,\partial_t \tilde{\rho}_\e \in L^1(\R^d \times (0,T))$. To summarize, we have shown 
\begin{equation}
    \Div (\beta(\tilde{\rho}_\e)\tilde{U}) = \beta(\tilde{\rho}_\e) \div \tilde{u} + \beta'(\tilde{\rho}_\e)\Div(\tilde{\rho}\tilde{U})_\e+\beta'(\tilde{\rho}_\e) \left(\tilde{u} \cdot \nabla \tilde{\rho}_\e - \div(\tilde{\rho}\tilde{u})_\e\right) \in L^1(\R^d \times (0,T)). 
\end{equation}
Now, let $O'\subset O$ be an open set. Pick any $\varphi\in C^0_c(O')$, $|\varphi|\leq 1$. Since $\beta(\tilde{\rho}_\eps) \tilde{U} \rightarrow\beta(\tilde{\rho})\tilde{U}$ in $L^1_{\rm loc}(\R^d \times [0,T])$ and $\div \tilde u, \Div (\tilde \rho \tilde U) \in \mathcal{M}(\R^d \times (0,T))$,  we deduce 
\begin{align}
    &\abs{\int  \varphi \,d\Div(\beta(\tilde \rho)\tilde U )} = \lim_{\e \to 0} \abs{ \int \varphi \,d\Div(\beta(\tilde{\rho}_\e)\tilde{U}) }
    \\ &  \leq C \left( \abs{\div \tilde{u}}(O') + \abs{\Div(\tilde{\rho} \tilde{U})} (O') + \int_0^T \limsup_{\e \to 0} \norm{\tilde{u}_t \cdot \nabla (\tilde{\rho}_t)_\eps-\div(\tilde{\rho}_t \tilde{u}_t)_\eps}_{L^1((\spt \varphi_t)_\e)} \, dt \right). 
\end{align}
The commutator is estimated by \cref{L:commutator} as
\begin{align}
    \int_0^T \limsup_{\eps\to 0} \norm{\tilde u_t \cdot \nabla  (\tilde\rho_t)_\eps-\div(\tilde\rho_t \tilde u_t)_\eps}_{L^1((\spt \varphi_t)_\e)} \, dt & \leq C \int_0^T \norm{\tilde\rho_t}_{L^{\infty}(\R^d)} \, \abs{\nabla \tilde u_t}\left( \spt \varphi_t\right)\, dt 
    \\ & \leq C \norm{\tilde \rho}_{L^\infty(\Omega \times (0,T) )} \left(\abs{\nabla \tilde u_t}\otimes dt\right)\,\left(O'\right). 
\end{align}
To summarize, since $\varphi\in C^0_c(O')$ is arbitrary, we deduce
\begin{equation}
\abs{\Div(\beta(\tilde{\rho})\tilde{U})} \leq C \left(\abs{\div \tilde{u}}+\abs{\Div(\tilde{\rho} \tilde{U})}+ \ \abs{\nabla \tilde u_t}\otimes dt \right)
\end{equation}
as measures in $O$, and \eqref{eq: div renormalized does not weight lambda} follows by \eqref{E:zero mass boundary} and \eqref{E:no mass 2}. 

\textsc{\underline{Step 4}: Conclusion.} By \eqref{eq: div renormalized does not weight lambda} and \eqref{eq: formula gluing} we infer that 
\begin{equation}
\tr_n\left(\beta(\tilde{\rho})\tilde{U}; (\partial \Omega\times(0,T) ) \cap O  \right) = -\tr_n \left(\beta(\tilde{\rho})\tilde{U}; (\partial \Omega^c\times(0,T)) \cap O \right).
\end{equation}
Therefore, exploiting Step 0 on the set $O\cap \Omega^c\times(0,T)$, we conclude    
\begin{align}
   \tr_n(\beta(\rho)u) & = \tr_n\left(\beta(\tilde{\rho})\tilde{U};(\partial \Omega\times(0,T)) \cap O \right) 
   \\ & = -\tr_n\left(\beta(\tilde{\rho})\tilde{U};(\partial \Omega^c\times(0,T))\cap O \right) 
   \\ & = -\tr_n\Big(\beta(\sigma)V;(\partial \Omega^c\times(0,T))\cap O \Big)
   \\ & = - \beta\left( \frac{\tr_n(\sigma V;(\partial \Omega^c\times(0,T))\cap O)}{\tr_n(V;(\partial \Omega^c\times(0,T))\cap O)} \right) \tr_n\Big(V;(\partial \Omega^c\times(0,T))\cap O\Big) 
   \\ & = - \beta\left( \frac{\tr_n\left(\tilde{\rho}\tilde{U};(\partial \Omega^c\times(0,T))\cap O\right)}{\tr_n\left(\tilde{U};(\partial \Omega^c\times(0,T))\cap O\right)} \right) \tr_n\left(\tilde{U};(\partial \Omega^c\times(0,T))\cap O\right) 
   \\ & = \beta\left( \frac{\tr_n(\rho u)}{\tr_n(u)} \right) \tr_n(u).
\end{align}
\end{proof}

We are ready to state and prove the main result of this section. 

\begin{theorem} \label{T:CE particular}
Let $\Omega\subset \R^d$ be a bounded open set with Lipschitz boundary and let $u \in L^\infty(\Omega\times (0,T))\cap L^1_{\rm loc}([0,T);BV_{\rm loc}(\Omega))$ be a vector field such that $\div u \in L^1(\Omega\times (0,T))$. Let $\Gamma^-,\Gamma^+\subset \Lambda$ be as in \cref{D:Gamma plus and minus} and assume that $u$ satisfies the following conditions:
\begin{itemize}
    \item[(i)] there exists an open set $O \subset \R^{d}\times (0,T)$ such that $\Gamma^- \subset O$, $u_t \in BV_{  \rm loc}(O_t)$ for a.e. $t\in (0,T)$ and $ \nabla u_t \otimes dt \in \mathcal{M}_{\rm loc} (O)$; 
    \item[(ii)] for a.e. $t\in (0,T)$ it holds  
    \begin{equation}
    \label{hp: gamma out exit section}
    \lim_{r\rightarrow 0} \frac{1}{r} \int_{(\Gamma^+_t)_{r}^{\rm in}} (u_t\cdot \nabla d_{\partial \Omega})_+ \,dx=0. 
    \end{equation}
\end{itemize}
Assume moreover that $c,f\in L^1(\Omega\times (0,T))$, $\rho_0 \in L^\infty (\Omega)$ and $g \in L^\infty(\Gamma^-)$.  Let $\rho \in L^\infty(\Omega\times (0,T))$ be any distributional solution to \eqref{CE} in the sense of \cref{D:Cont Eq weak sol}. Then, for any  $\beta\in C^1(\R)$, it holds
\begin{align}
&\int_0^T \int_\Omega \beta(\rho ) \Big( \partial_t \varphi  + u\cdot \nabla \varphi\Big) \,dx \, dt + \int_0^T \int_\Omega \Big(\left(\beta(\rho)- \rho \beta'(\rho)\right)\div u + (c\rho + f)\beta'(\rho) \Big)\varphi\,dx\,dt\\
&= - \int_\Omega \beta(\rho_0) \varphi(x,0) \,dx  + \int_{\Gamma^-}\varphi\beta(g) \tr_n(u)\,d\mathcal L^T_{\partial \Omega}+\langle \mathcal{B}_\beta [u,\rho],\varphi \rangle, \label{ren_formula}
\end{align}
 for all $\varphi\in C^1_c(\overline \Omega\times [0,T))$, where 
$$
\langle \mathcal{B}_\beta [u,\rho],\varphi\rangle= \lim_{r\rightarrow 0} \frac{1}{r}\int_0^T \int_{(\partial \Omega)_{r}^{\rm in}}\varphi \beta(\rho) (u_t\cdot \nabla d_{\partial \Omega})_-  \,dx \, dt.
$$
Moreover, if we assume in addition that $(\div u-2c)_-\in L^1((0,T);L^\infty(\Omega))$, then $\rho$ is unique.
\end{theorem}

\begin{proof}
To improve readability, the proof will be divided into steps.

\textsc{\underline{Step 0}:  Uniqueness.} We start by proving that the renormalization formula \eqref{ren_formula} implies uniqueness. By linearity it is enough to prove that $\rho\equiv 0$ whenever $\rho_0=g=f=0$.  Let $\alpha\in C^\infty_c([0,T))$ be arbitrary. By  choosing $\varphi(x,t)=\alpha(t)$ in \eqref{ren_formula} we obtain
\begin{align}
\int_0^T \alpha'\left(\int_\Omega \beta(\rho ) \,dx\right) dt + \int_0^T \alpha \left(\int_\Omega \left(\beta(\rho)- \rho \beta'(\rho)\right)\div u + c\rho \beta'(\rho) \,dx\right) dt= \langle  \mathcal{B}_\beta[u,\rho],\alpha\rangle.
\end{align}
Since the sequence of functions
$$
t\mapsto \frac{1}{r} \int_{(\partial \Omega)_{r}^{\rm in}} \beta(\rho) (u_t\cdot \nabla d_{\partial \Omega})_-  \,dx
$$
is bounded in $L^\infty((0,T))$, by weak* compactness we can find $b_\beta\in L^\infty((0,T))$ such that 
$$
\langle  \mathcal{B}_\beta[u,\rho],\alpha\rangle=\int_0^T \alpha b_\beta \,dt.
$$
In particular, by choosing $\beta(s)=s^2$, we deduce that the function $F(t):=\int_\Omega |\rho(x,t)|^2\,dx$ belongs to $W^{1,1}((0,T))$ with $F(0)=0$. Note that for any $\beta\geq 0$ it must hold $b_\beta\geq 0$. Thus, for $\alpha \geq 0$ we can split and bound 
\begin{align}
    \int_0^T \alpha' F\,dt &=\int_0^T \alpha\left( \int_\Omega |\rho|^2 ( \div u - 2c)\,dx \right) dt + \int_0^T\alpha b_\beta\,dt\\
    &\geq -\int_0^T \alpha\left( \int_\Omega |\rho|^2  (\div u-2c)_- \,dx\right)dt \geq -\int_0^T \alpha G F\,dt,
\end{align}
with $G(t):=\| (\div u-2c)_- (t)\|_{L^\infty(\Omega)} \in L^1((0,T))$.
For a.e. $t\in (0,T)$, we let $\alpha$ converge to the characteristic function of the time interval $[0,t]$ and deduce $F(t) \leq \int_0^t G(s)F(s)\,ds$,
from which $F\equiv 0$  follows by Gr\"onwall inequality.

We are left to prove the validity of \eqref{ren_formula}, whose proof will be divided into three more steps.

\textsc{\underline{Step 1}:  Interior renormalization.} We set 
$$
H := \left(\beta(\rho)- \rho \beta'(\rho)\right)\div u + (c\rho + f)\beta'(\rho) \in  L^1(\Omega\times(0,T)).
$$
By the renormalization property for $BV$ vector fields \cite{Ambr04}, we have  
\begin{equation}
    \int_0^T \int_\Omega   \left( \beta(\rho ) \Big( \partial_t \psi  + u\cdot \nabla \psi\Big) + H \psi \right)  dx \, dt  = - \int_\Omega\beta(\rho_0) \psi (x,0) \, dx \qquad     \forall \psi\in C^\infty_c(\Omega\times[0,T)). \label{eq: renormalized equation 1}
\end{equation}

By a standard density argument, \eqref{eq: renormalized equation 1} can be tested with Lipschitz functions vanishing on $\Lambda$. For any $r>0$ we set $\chi_r = \frac{1}{r}d_{\partial \Omega} \wedge 1$ and, given $\varphi\in C^1_c(\overline{\Omega}\times[0,T))$, we get 
\begin{align}
    \int_0^T \int_\Omega \chi_r \left( \beta(\rho ) \Big( \partial_t \varphi  + u\cdot \nabla \varphi\Big) + H \varphi \right)  dx \, dt &+ \int_0^T \int_\Omega \varphi \beta(\rho) u\cdot \nabla \chi_r \,dx\,dt \\
    &= - \int_\Omega \chi_r \beta(\rho_0) \varphi (x,0) \, dx. \label{eq: CE tested with chi_e}
\end{align}
By dominated convergence, we have that
\begin{align}
    \lim_{r \to 0}  \int_0^T \int_\Omega \chi_r \left( \beta(\rho ) \Big( \partial_t \varphi + u\cdot \nabla \varphi\Big) + H \varphi \right) dx\,dt &= \int_0^T \int_\Omega \left( \beta(\rho ) \Big( \partial_t \varphi  + u\cdot \nabla \varphi\Big) + H \varphi \right) \,dx\,dt  
    \\   \lim_{r\rightarrow 0} \int_\Omega \chi_r \beta(\rho_0) \varphi(x,0) \,dx &= \int_\Omega \beta(\rho_0) \varphi (x,0) \,dx. 
\end{align} 
We study the limit of the term involving $\nabla \chi_r$. Since $u$ behaves differently around $\Gamma^+$ and $\Gamma^-$, we split the integral as follows. Since $\mathbbm{1}_{\Gamma^+}\in L^\infty(\Lambda)$, by Gagliardo's theorem and a truncation argument, we find $\lambda^+\in W^{1,1}(\Omega\times(0,T))$ such that $0\leq \lambda^+\leq 1$ and its trace on $\Lambda$ satisfies $(\lambda^+)^\Lambda=\mathbbm{1}_{\Gamma^+}$. Then, set $\lambda^- := 1- \lambda^+$, that has the same properties relatively to $\Gamma^-$. We also define $\varphi^\pm := \varphi \ \lambda^\pm$. We claim that 
\begin{equation}
    \lim_{r\to 0} \int_0^T \int_\Omega \varphi^- \beta(\rho) u\cdot \nabla \chi_r \,dx\,dt = -\int_{\Gamma^-} \varphi\beta(g) \tr_n(u) \,d\mathcal L^T_{\partial \Omega}   \label{eq: behaviour on Gamma-}
\end{equation}
and
\begin{equation}
-\lim_{r \to 0}\int_0^T \int_\Omega \varphi^+ \beta(\rho) u\cdot \nabla \chi_r \,dx\,dt =\lim_{r \to 0}\frac{1}{r} \int_0^T \int_{(\partial \Omega)_r^{\rm in}} \varphi \beta(\rho) (u \cdot \nabla d_{\partial \Omega})_- \, dx \, dt=:\langle  \mathcal{B}_\beta[u,\rho],\varphi\rangle. \label{eq: behaviour on Gamma+}
\end{equation}
Note that the limit \eqref{eq: behaviour on Gamma+} is the only one which is not known to exist a priori. However, the fact that all the other terms in \eqref{eq: CE tested with chi_e} have a finite limit, shows that also the limit in \eqref{eq: behaviour on Gamma+} is finite and thus it defines the linear operator $\mathcal{B}_\beta [u,\rho]$ acting on smooth test functions $\varphi$.

\textsc{\underline{Step 2}: Behaviour on $\Gamma^-$. } 
We now check \eqref{eq: behaviour on Gamma-}. Let $V:= \varphi^-\beta(\rho) U$. Since $\varphi^- \in W^{1,1} \cap L^\infty(\Omega \times (0,T))$ and $\beta(\rho) U \in \mathcal{MD}^{\infty}(\Omega \times (0,T))$, by \cref{l: product rule}, we infer that  $V \in \mathcal{MD}^{\infty}(\Omega\times(0,T))$. Since $\chi_r$ does not depend on time and $\chi_r\big|_{\partial \Omega}\equiv 0$, we have\footnote{Note that, since $\varphi^-$ is a Sobolev function, \eqref{eq: renormalized equation} implies $\Div V\in L^1(\Omega\times (0,T))$.} 
\begin{align}
    \int_{\Omega\times \left(\{0\}\cup  \{T\}\right)}\chi_r \tr_n(V;\partial (\Omega\times (0,T)))\,dx & = \int_{\partial (\Omega\times (0,T))} \chi_r \tr_n(V;\partial (\Omega\times (0,T))) \,d\mathcal H^d \\
    &=\int_0^T \int_\Omega \varphi^-\beta(\rho) u\cdot\nabla \chi_r \,dx\,dt+\int_0^T\int_\Omega \chi_r \Div V\,dx\,dt.
\end{align}
Letting $r\rightarrow 0$ we get 
\begin{align}
    \int_{\Omega\times \left(\{0\}\cup  \{T\}\right)} &\tr_n(V;\partial (\Omega\times (0,T)))\,dx =\lim_{r\rightarrow 0} \int_0^T \int_\Omega \varphi^-\beta(\rho) u\cdot\nabla \chi_r \,dx\,dt+\int_0^T\int_\Omega  \Div V\,dx\,dt\\
    &=\lim_{r\rightarrow 0} \int_0^T \int_\Omega \varphi^-\beta(\rho) u\cdot\nabla \chi_r \,dx\,dt  +\int_{\partial (\Omega\times (0,T))}  \tr_n(V;\partial (\Omega\times (0,T))) \,d\mathcal H^d.
\end{align}
In the above formula the boundary integrals on $\Omega\times\left(\{0\}\cup \{T\}\right)$ cancel. Thus, since $\Lambda=\partial \Omega\times (0,T)$, the above identity is equivalent to
\begin{align} 
    \lim_{r\to 0} \int_0^T \int_\Omega \varphi^- \beta(\rho) u\cdot \nabla \chi_r \,dx\,dt &=- \int_\Lambda \tr_n(\varphi^- \beta(\rho) U; \partial(\Omega\times(0,T)))\big|_{\Lambda} \,d\mathcal L^T_{\partial \Omega}.\label{eq: limit on Gamma-}
\end{align}
Moreover, by \eqref{eq: product rule trace} and the chain rule for traces by \cref{P:chain-rule} (recall that $\rho = g$ on $\Gamma^-$ in the sense of \cref{D:Cont Eq weak sol}), on $\Gamma^-$ we have that
\begin{align}
    \tr_n(\varphi^- \beta(\rho) U; \partial(\Omega\times(0,T))) &= (\varphi^-)^{\Omega\times(0,T)}\tr_n(\beta(\rho) U; \partial(\Omega\times(0,T))) \\
    &= \varphi\mathbbm{1}_{\Gamma^-} \tr_n(\beta(\rho) u)= \varphi\mathbbm{1}_{\Gamma^-} \beta(g) \tr_n(u),
\end{align}
where in the second equality we have used \eqref{eq: tr beta(rho) u}. In the above identity $(\varphi^-)^{\Omega\times(0,T)}$ denotes the strong  trace of $\varphi^-$ on $\partial (\Omega\times(0,T))$ in the sense of Sobolev functions (see \cref{T:trace_in_BV}).
This proves \eqref{eq: behaviour on Gamma-}. 

\textsc{\underline{Step 3}: Behaviour on $\Gamma^+$. } We check \eqref{eq: behaviour on Gamma+}. For a.e. $t \in (0,T)$ we estimate 
\begin{align}
    \frac{1}{r} & \abs{ \int_{(\partial \Omega)_r^{\rm  in} } \varphi^+ \beta(\rho) (u\cdot \nabla d_{\partial \Omega} ) \,dx   +  \int_{ (\partial \Omega)_r^{\rm  in} } \varphi \beta(\rho) (u\cdot \nabla d_{\partial \Omega})_- \,dx } 
    \\ & \leq \frac{1}{r} \int_{(\partial \Omega)_r^{\rm in}} \abs{\varphi^+ \beta(\rho)} (u \cdot \nabla d_{\partial \Omega})_+  \, dx + \frac{1}{r} \int_{(\partial \Omega)_r^{\rm \rm in}} \abs{\varphi^- \beta(\rho) } (u \cdot \nabla d_{\partial \Omega})_- \, dx.  
\end{align}
To estimate the first term we split
\begin{align}
    \frac{1}{r} \int_{(\partial\Omega)_r^{\rm in}}\abs{\varphi^+ \beta(\rho)} (u\cdot \nabla d_{\partial \Omega})_+ \,dx  \leq C \left(\frac{1}{r}  \int_{(\Gamma^+_t)_r^{ \rm in}} (u\cdot \nabla d_{\partial \Omega})_+ \,dx
  + \frac{1}{r} \int_{(\partial\Omega)_r^{\rm in}\setminus (\Gamma^+_t)_r^{\rm  in}} \abs{\varphi^+} \,dx\right). 
\end{align}
The first term goes to $0$ as $r \to 0$ for a.e. $t \in (0,T)$ by \eqref{hp: gamma out exit section}. To estimate the second term, by the slicing properties for Sobolev functions (see \cref{p: slicing and traces}), it follows that $\abs{\varphi^+(\cdot, t)} \in W^{1,1}(\Omega)$ for a.e. $t \in (0,T)$ and $\abs{\varphi^+(\cdot, t)}^{\Omega} = \abs{\varphi(\cdot, t)}\mathbbm{1}_{\Gamma^+_t}$ as functions in $L^\infty(\partial \Omega)$. Hence, for a.e. $t \in (0,T)$, recalling that $\Gamma_t^+$ is closed, by \cref{p: restriction of traces sobolev} we deduce
\begin{align}
    \lim_{r\to  0} \frac{1}{r}\int_{(\partial\Omega)_r^{\rm in}\setminus (\Gamma^+_t)_r^{\rm in}} \abs{\varphi^+} \, dx & = \lim_{r\to  0} \frac{1}{r}\int_{(\partial\Omega)_r^{\rm in}} \abs{\varphi^+} \,dx - \lim_{r\to  0} \frac{1}{r}\int_{(\Gamma^+_t)_r^{\rm  in}} \abs{\varphi^+} \,dx 
    \\ & = \int_{\partial\Omega} \abs{\varphi^+(\cdot, t)}^\Omega \,d\H^{d-1} - \int_{\Gamma^+_t} \abs{\varphi^+(\cdot, t)}^\Omega \,d\H^{d-1} 
    \\ & = \int_{\Gamma^-_t} \abs{\varphi(\cdot, t) } \mathbbm{1}_{\Gamma^+_t} \,d\H^{d-1} = 0.
\end{align}
For the term involving $(u \cdot \nabla d_{\partial \Omega})_-$, given $\delta>0$, we choose $K\subset \Gamma^-_t$ compact and $A\subset \partial \Omega$ open such that $K\subset A \subset\subset \Gamma^-_t$ and $\H^{d-1}(\Gamma^-_t\setminus K)\leq\delta$. Then  we estimate 
\begin{align}
    \frac{1}{r} \int_{(\partial \Omega)_r^{\rm in}} \abs{\varphi^- \beta(\rho)} (u \cdot \nabla d_{\partial \Omega})_- \, dx  
    \leq C\left( \frac{1}{r} \int_{(\partial\Omega \setminus A)_r^{\rm in}} \abs{\varphi^-} \, dx   +  \frac{1}{r} \int_{(\overline{A})_r^{\rm in} } (u \cdot \nabla d_{\partial \Omega})_- \, dx\right).
\end{align}
Since $\overline{A}\subset\Gamma^-_t$ is closed, the second term vanishes in the limit  by \cref{P: BV and exit}. Since also $\partial\Omega\setminus A$ is closed in $\partial\Omega$, using \cref{p: slicing and traces} and \cref{p: restriction of traces sobolev} as before, it is readily checked that
\[
    \lim_{r\to 0} \frac{1}{r} \int_{(\partial\Omega \setminus A)_r^{\rm in}} \abs{\varphi^-} \, dx =  \int_{\partial\Omega \setminus A}\abs{\varphi} \mathbbm{1}_{\Gamma^-_t} \, d\H^{d-1}
    \leq \norm{\varphi}_{L^\infty(\partial\Omega)} \H^{d-1}(\Gamma^-_t\setminus K) \leq C \delta.
\]
The arbitrariness of $\delta>0$ implies that the above limit vanishes. To summarize, we have proved
\begin{equation}
    \lim_{r \to 0} \frac{1}{r} \abs{ \int_{(\partial \Omega)_r^{\rm in} } \varphi^+ \beta(\rho) (u\cdot \nabla d_{\partial \Omega} ) \,dx  + \int_{ (\partial \Omega)_r^{\rm in} } \varphi \beta(\rho) (u\cdot \nabla d_{\partial \Omega})_- \,dx } = 0 \quad \text{ for a.e. } t \in (0,T). 
\end{equation}
Moreover, by \cref{p: minkowski federer}, we find $r_0>0$ such that for all $r< r_0$ we have
\begin{align}
    \frac{1}{r}   \int_{(\partial \Omega)_r^{\rm in} } \abs{\varphi^+ \beta(\rho) (u\cdot \nabla d_{\partial \Omega} )  +  \varphi \beta(\rho) (u\cdot \nabla d_{\partial \Omega})_-} \,dx 
     \leq C\frac{\mathcal{H}^d((\partial \Omega)_r)}{r}  \leq 2 C  \mathcal{H}^{d-1}(\partial \Omega)
\end{align}
for a.e. $t \in (0,T)$. Then, \eqref{eq: behaviour on Gamma+} follows by dominated convergence.
\end{proof}

A direct consequence of \cref{T:CE particular} is the following result, in the specific setting of a transport equation with a divergence-free vector field which is tangent (in the Lebesgue sense) to the boundary. It should be compared with \cite{DRINV23}*{Theorem 1.3}.
\begin{corollary}\label{C: TE tangent}
   Let $\Omega\subset \R^d$ be a bounded open set with Lipschitz boundary and $u \in L^\infty(\Omega\times (0,T))\cap L^1_{\rm loc}([0,T);BV_{\rm loc}(\Omega))$ be a given vector field such that $\div u = 0$ and $u_n^{\partial \Omega} (\cdot,t) \equiv 0$\footnote{Note that this implies $\Gamma^-=\emptyset$.} for a.e. $t\in (0,T)$.
Then, for any $\rho_0\in L^\infty(\Omega)$, there exists a unique distributional solution $\rho \in L^\infty(\Omega\times (0,T))$ to 
\begin{equation}\label{Transp Eq}
    \left\{
    \begin{array}{rcll}
        \partial_t \rho +u  \cdot \nabla \rho &=& 0 & \text{ in }\Omega\times (0,T) \\[1ex]
        \rho(\cdot,0) &=&  \rho_0 & \text{ in }  \Omega,
    \end{array}
    \right.
\end{equation}
in the sense of \cref{D:Cont Eq weak sol}. Moreover, for any  $\beta\in C^1(\R)$, it holds
$$
\int_0^T \int_\Omega \beta(\rho ) \Big( \partial_t \varphi  + u\cdot \nabla \varphi\Big) \,dx \, dt = - \int_\Omega \beta(\rho_0) \varphi(x,0) \,dx   \qquad \forall \varphi\in C^1_c(\overline \Omega\times [0,T)),
$$
and thus also
$$
\int_\Omega \beta(\rho(x,t))\,dx= \int_\Omega \beta(\rho_0(x))\,dx \qquad \text{for a.e. } t\in (0,T).
$$
\end{corollary}

\subsection{A counterexample to uniqueness with normal Lebesgue trace}\label{S:counter ex uniq}
In this section we prove \cref{P:counter-ex}. The proof relies on \cite{CDS14}*{Proposition 1.2}, which is a suitable modification of the celebrated construction by Depauw in \cite{De03}.

\begin{proof}[Proof of \cref{P:counter-ex}]
Let $b : \R^2 \times (0,1) \to \R^2 $ be the time-dependent vector field built in \cite{De03} satisfying the following properties:
\begin{itemize}
    \item $b \in L^\infty(\R^2 \times (0,1))$; 
    \item for every $t \in (0,1)$, $b(\cdot, t): \R^2 \to \R^2$ is piecewise smooth, divergence-free and $b(\cdot, t)\in BV_{\rm loc} (\R^2)$; 
    \item $b \in L^1_{\rm loc} ((0,1); BV_{\rm loc}(\R^2))$, but $b \notin L^1([0,1); BV_{\rm loc}(\R^2))$, namely the $BV$ regularity blows up as $t \to 0^+$ in a non-integrable way;
    \item the Cauchy problem \eqref{CE counter ex} with initial datum $\rho_0=0$ admits a nontrivial bounded solution. 
\end{itemize}
Following \cite{CDS14}*{Proposition 1.2}, we adopt the notation $(y,r) \in  \R^2\times (0,+\infty) =:\Omega$ and define the autonomous vector field $u: \Omega \to \R^3$ as
\begin{equation}
    u(y,r) = \begin{cases}
        (b(y,r),1) & r \in (0,1), 
        \\ (0,1) & r \geq 1. 
    \end{cases} 
\end{equation}
In other words, to construct $u$ we are lifting from $2$-d to $3$-d the Depauw vector field, turning the time variable into a third space variable.
By \cite{CDS14}*{Proposition 1.2}, $u$ satisfies the following properties:
\begin{itemize}
    \item $u \in L^\infty(\Omega)\cap BV_{\rm loc}(\Omega)$ and $u$ is divergence-free; 
    \item the initial-boundary value problem \eqref{CE counter ex} admits infinitely many different solutions in the sense of \cref{D:Cont Eq weak sol}. 
\end{itemize}
Since $d_{\partial \Omega}(y,r) = r$, it is clear that $u_n^{\partial \Omega}(y,0) = -1$ for all $y \in \R^2$. Moreover, by \cref{T:strong trace equals weak}, it must hold $ \tr_n(u; \partial \Omega) \equiv u_n^{\partial \Omega}= -1$.
\end{proof}

\begin{appendix}

\section{Measurability of the space-time trace} 

In the following lemma we study the existence of the full trace of a time-dependent vector field with spatial $BV$ regularity on a portion of the boundary of a space-time cylinder. For the reader's convenience we give a complete the proof, which is based on that of \cite{ACM05}*{Proposition 3.2}. 

\begin{lemma} \label{l: measurability of space-time trace}
Let $u$ be a vector field as in \cref{T:CE general}. Then, there exists a function $\tr(u)\in L^\infty(\Lambda\cap O)$ such that $\tr(u) (\cdot,t) = u_t^\Omega$ for a.e. $t \in (0,T)$ as functions in $L^\infty(O_t \cap \partial \Omega)$. 
\end{lemma}

\begin{proof}
For any component $u^i$, we define its distributional trace on $\Lambda \cap O$ by setting 
\begin{equation}
    \langle \tr(u^i),\varphi \rangle := \int_{O_{\rm in}} u^i \div\varphi \,dx\,dt + \int_{O_{\rm in}} \varphi \cdot \,d \nabla u_t^i \otimes dt \qquad  \forall \varphi\in C^1_c(O;\R^d),
\end{equation}
where $O_{\rm in}=O\cap(\Omega\times(0,T))$. We claim that there exists $g^i \in L^\infty(\Lambda \cap O)$ such that 
\begin{equation}
     \langle \tr(u^i),\varphi \rangle = \int_{\Lambda\cap O} g^i \ n_\Omega \cdot \varphi \, d\mathcal L^T_{\partial \Omega} \qquad \forall \varphi \in C^1_c(O; \R^d). \label{eq: trace of u^i}
\end{equation}
Indeed, for any $\varphi\in C^1_c(O;\R^d)$, following \cite{ACM05}*{Lemma 3.1}, we find $\varphi_\eps\in C^1_c(O;\R^d)$ such that
\begin{itemize}
    \item $\varphi_\eps =\varphi$ on $(\Lambda)_\eps \cap O_{\rm in}$ and $\varphi_\eps\equiv 0$ on $O_{\rm in}\setminus (\Lambda)_{2\eps}$,
    \item $\norm{\varphi_\eps}_{L^\infty(O)} \leq \norm{\varphi}_{L^\infty(O)}$,
    \item $\int_{O_{\rm in}} \abs{\div \varphi_\eps} \, dx\,dt \leq \int_{\Lambda\cap O} \abs{\varphi \cdot n_\Omega} \, d\mathcal L^T_{\partial \Omega} + \eps$.
\end{itemize}
It is immediate to see that the distribution $\tr(u^i)$ is supported on $\Lambda\cap O$. Hence, noticing that $\spt(\varphi_\e) \subset \spt(\varphi)$, we write
\begin{align}
    \abs{\langle \tr(u^i),\varphi \rangle} &= \abs{\langle \tr(u^i),\varphi_\eps \rangle}  \leq \abs{\int_{O_{\rm in}} u^i \div \varphi_\eps \,dx\,dt} + \abs{\int_{(\Lambda)_{2\eps}  \cap O_{\rm in}} \varphi_\eps \cdot \,d\nabla u_t^i \otimes dt}  \\
    &\leq \norm{u^i}_{L^\infty(\R^d)} \int_{O_{\rm in}} |\div\varphi_\eps| \, dx\,dt + \norm{\varphi_\eps}_{L^\infty(O)} \abs{\nabla u_t^i}\otimes dt \, \left( (\Lambda)_{2\eps} \cap O_{\rm in} \cap \spt(\varphi) \right)  \\
    &\leq \norm{u^i}_{L^{\infty}(\R^d)} \left(\int_{\Lambda\cap O} |\varphi\cdot n_\Omega| \, d\H^d + \eps \right)  + \norm{\varphi}_{L^\infty(O)} \abs{\nabla u_t^i}\otimes dt \, \left( (\Lambda)_{2\eps}  \cap O_{\rm in} \cap \spt(\varphi) \right).
\end{align}
Since $\bigcap_{\e>0} (\Lambda)_\eps \cap O_{\rm in} = \emptyset$ and $\abs{\nabla u_t^i}\otimes dt \in \mathcal{M}_{\rm loc}(O)$, letting $\e \to 0$, we obtain 
\[
   \abs{\langle \tr(u^i),\varphi \rangle} \leq \norm{u^i}_{L^\infty(\R^d)} \int_{\Lambda\cap O} |\varphi\cdot n_\Omega| \, d\mathcal L^T_{\partial \Omega}  \qquad \forall \varphi \in C^1_c(O;\R^d).
\]
Hence, there exists $T^i\in L^\infty(\Lambda\cap O;\R^d)$ such that 
    \[
        \langle \tr(u^i),\varphi \rangle = \int_{\Lambda\cap O} T^i\cdot \varphi \, d\mathcal L^T_{\partial \Omega}  \qquad \forall \varphi \in C^1_c(O; \R^d).
    \]
    Moreover, it holds
    \[
        \abs{\int_{\Lambda\cap O} T^i\cdot \varphi \, d\mathcal L^T_{\partial \Omega}} \leq \norm{u^i}_{L^\infty(\R^d)} \int_{\Lambda\cap O} |\varphi\cdot n_\Omega| \, d\mathcal L^T_{\partial \Omega} \qquad \forall \varphi \in L^1(\Lambda\cap O; \R^d).
    \]
Hence, we find $g^i \in L^\infty(\Lambda \cap O)$ such that $T^i= g^i n_{\Omega}$, thus proving \eqref{eq: trace of u^i}. We define $\tr(u)\in L^\infty(\Lambda\cap O;\R^d)$ such that $(\tr(u))^i=\tr(u^i)$ for all $i=1,\dots,d$. 

To conclude, we check that $\tr(u)(\cdot, t)$ agrees with the $BV$ trace of $u(\cdot, t)$ for a.e. $t$ as $L^\infty$ functions on $O_t \cap \partial \Omega$. Indeed, letting $\varphi\in C^1_c(O;\R^d)$, by Fubini's theorem we compute
    \begin{align}
        \int_{\Lambda\cap O} \tr(u^i) \varphi \cdot n_\Omega \,dx\,dt & = \int_{O_{\rm in}} u^i \div\varphi \,dx\,dt + \int_{O_{\rm in}} \varphi \cdot \,d \nabla u_t^i\otimes dt  \\
        & = \int_0^T \left( \int_{O_t\cap\Omega} u_t^i \div\varphi \,dx + \int_{O_t\cap\Omega} \varphi \cdot \,d\nabla u_t^i\right) dt  \\
        &= \int_0^T \left( \int_{O_t\cap\partial\Omega} (u_t^i)^{O_t\cap\Omega} \varphi \cdot n_\Omega\,d\H^{d-1} \right) dt. 
    \end{align}
    Thus, by a standard density argument, we have that 
    \[
        \int_0^T \left( \int_{O_t\cap\partial\Omega} \left(\tr(u^i)(\cdot,t) -(u^i_t)^{O_t\cap\Omega}\right)\ \varphi \cdot n_\Omega \,d\H^{d-1} \right) dt=0 \qquad  \forall \varphi\in L^1(\Lambda\cap O;\R^d). 
    \]
The conclusion follows since $\varphi\cdot n_\Omega$ can be chosen to be any scalar $L^1$ function. 
\end{proof}

\end{appendix}

\bibliographystyle{plain} 
\bibliography{biblio}

\end{document}